\documentclass[12pt]{article}
\usepackage{amsmath}
\usepackage{amssymb}
\usepackage{amsthm}
\usepackage[top=40truemm, left=30truemm, right=30truemm]{geometry}
\usepackage{tikz}
\usetikzlibrary{positioning}
\usepackage{amscd}
\newtheorem{dfn}{Definition}[section]
\newtheorem{thm}[dfn]{Theorem}
\newtheorem{lem}[dfn]{Lemma}
\newtheorem{rem}[dfn]{Remark}
\newtheorem{cor}[dfn]{Corollary}

\newtheorem{prop}[dfn]{Proposition}
\newtheorem{ex}[dfn]{Example}
\newtheorem{que}[dfn]{Question}
\usepackage{amscd}
\usepackage{graphicx}
\usepackage[all]{xy}
\usepackage{tikz-cd} 

\title{The Reciprocal Kirchberg Algebras}
\author{Taro Sogabe}
\date{}
\begin{document} 
\maketitle
\abstract
For two unital Kirchberg algebras with finitely generated K-groups,
we introduce a property, called reciprocality, which is proved to be closely related to the homotopy theory of Kirchberg algebras.
We show the Spanier--Whitehead duality for bundles of separable nuclear UCT C*-algebras with finitely generated K-groups and conclude that two reciprocal Kirchberg algebras share the same structure of their bundles.
\tableofcontents
\section{Introduction}
Bundles of C*-algebras naturally appear in several areas, including Kasparov's KK-theory for continuous fields of C*-algebras \cite{K, D3}, the classification of group actions via topological invariants introduced in M. Izumi and H. Matui's work \cite{IM}, the Dixmier--Douady and the Dadarlat--Pennig theory \cite{DD, DP}.
If we restrict ourselves to considering locally trivial bundles over $X$ whose fiber is a fixed C*-algebra $A$,
they are classified by the homotopy set $[X, \operatorname{BAut}(A)]$ consisting of homotopy classes of continuous maps from $X$ to the classifying space of the automorphism group $\operatorname{Aut}(A)$.
The homotopy set $[X, \operatorname{BAut}(A)]$ gives a cohomology group in the Dadarlat--Pennig theory,
and it also provides the topological invariant of group actions in \cite{IM} (i.e., $[\operatorname{B}G, \operatorname{BAut}(A)]$ is equal to the set of cocycle conjugacy classes of the outer actions of $G$ to the stable Kirchberg algebra $A$).
Thus, we are interested in the homotopy theory of $\operatorname{BAut}(A)$ including the homotopy groups $\pi_n(\operatorname{BAut}(A))\cong \pi_{n-1}(\operatorname{Aut}(A))$,
and, from the viewpoint of C*-algebra theory, it is an interesting problem to investigate $\pi_{n-1}(\operatorname{Aut}(A))$, $[X, \operatorname{Aut}(A)]$ and $[X, \operatorname{BAut}(A)]$ using the K-theory of $X$ and $A$,
where the Kirchberg algebras $A$ including the Cuntz algebras $\mathcal{O}_n, \mathcal{O}_\infty$ are appropriate targets to challenge the problem.

The research on the bundles of $\mathcal{O}_n$ was initiated by M. Dadarlat,
and, in \cite[Sec. 2]{D2}, he introduced a crucial idea of focusing on the Puppe sequences appearing in the construction of bundles of $\mathcal{O}_n$ as the Cuntz--Pimsner algebras.
In our previous work,
the above idea and the Dadarlat--Pennig theory lead us to \cite[Thm. 3.8]{S2} which gives a natural bijection $$[X,\operatorname{BAut}(\mathcal{O}_n)]\to [X, \operatorname{BAut}(\mathbb{M}_{n-1}(\mathcal{O}_\infty))].$$
By the above bijection,
the classification of the bundles of $\mathcal{O}_n$ in \cite{D2} via a combination of the Puppe sequence argument and Kirchberg--Gabe's theorem \cite[Thm. E, G]{G}  is now reduced to the classification of corners of $C(X)\otimes\mathcal{O}_\infty\otimes \mathbb{K}$,
and the later classification is easy.
In the point that a difficult classification problem is reduced to easier one,
the bijection is interesting,
and it is natural to ask why these two algebras have such a nice property concerning their bundles.
A notable property of the pair $\mathcal{O}_n$ and $\mathbb{M}_{n-1}(\mathcal{O}_\infty)$ is that they share the same homotopy groups
\[\pi_i(\operatorname{Aut}(\mathcal{O}_n))\cong \pi_i(\operatorname{Aut}(\mathbb{M}_{n-1}(\mathcal{O}_\infty))), \quad i\geq 1,\]
and the question is how to characterize unital Kirchberg algebras with the same homotopy groups.

The reciprocality defined via the Spanier--Whitehead K-duality gives a nice answer to the above question.
For a separable nuclear UCT C*-algebra $C$ with finitely generated K-groups,
one can find  another separable nuclear UCT C*-algebra $D(C)$ uniquely up to $KK$-equivalence,
which is called the Spanier--Whitehead K-dual of $C$.
This is a duality for $KK$-theory established in \cite{KP, KK, KS} by D. S. Kahn, J. Kaminker, I. Putnam, C. Schochet,
and the algebra $D(C)$ is characterized by $KK^i(C, \mathbb{C})\cong KK^i(\mathbb{C}, D(C))$.

In \cite{D4},
M. Dadarlat computes the groups $[X, \operatorname{Aut}(A)]$
using the mapping cone $C_{u_A}$ of the unital map $u_A : \mathbb{C}\to A$,
and obtains the group isomorphisms
$$\pi_i(\operatorname{Aut}(A))\cong KK^{i+1}(C_{u_A}, A),\quad i\geq 1,$$
for every unital UCT Kirchberg algebra $A$ (see \cite[Cor. 5.10.]{D4}).
If $A$ has finitely generated K-groups,
$C_{u_A}$ is dualizable and one has
$$\pi_i(\operatorname{Aut}(A))=K_{i+1}(D(C_{u_A})\otimes A),\quad i\geq 1.$$
Now the condition $\pi_i(\operatorname{Aut}(A))\cong \pi_i(\operatorname{Aut}(B)), i\geq 1$ is equivalent to the following $KK$-equivalence $$D(C_{u_A})\otimes A\sim_{KK}D(C_{u_B})\otimes B.$$
There are two solutions to the above equivalence relation and the nontrivial one is the definition of reciprocality.
\begin{dfn}\label{REC}
The unital UCT Kirchberg algebras $A$ and $B$ with finitely generated K-groups are called reciprocal if the following holds $\colon$
\[A\sim_{KK} D(C_{u_B}),\quad B\sim_{KK} D(C_{u_A}).\]
\end{dfn}
\begin{thm}[{Thm. \ref{MT}}]\label{sm}
Let $A$ and $B$ be two unital UCT Kirchberg algebras with finitely generated K-groups.

\end{thm}
\begin{rem}
It might be surprising that a dichotomy holds in statement 2 because a similar result does not hold in the case of stable Kirchberg algebras.
Statement 3 (and also our main result Thm. \ref{RB}) is interesting because it seems impossible to construct some nice maps between $\operatorname{Aut}(A)$ and $\operatorname{Aut}(B)$
\end{rem}
There is a categorical picture to understand the reciprocality.
Note that, for two dualizable algebras $C_1, C_2$ and a morphism $C_1\to C_2$, the duality
\[KK(C_1, C_2)\cong KK(D(C_2), D(C_1))\]
provides another morphism $D(C_1)\leftarrow D(C_2)$, where the dual morphism depends on the choice of the duality classes (see Def. \ref{SK}, Lem. \ref{b}).
We show that the mapping cone sequences of the reciprocal algebras are ``dual'' to each other.
\begin{thm}[{Thm. \ref{ca}}]\label{mn}
Let $A$ and $B$ be the reciprocal Kirchberg algebras.
For the mapping cone sequence
\[C_{u_A}\xrightarrow{e_A} \mathbb{C}\xrightarrow{u_A} A,\]
the Spanier--Whitehead K-duality gives another sequence

\[\xymatrix{
D(C_{u_A})& D(\mathbb{C})\ar[l]& D(A)\ar[l]\\
B\ar@{.>}[u]&\mathbb{C}\ar[l]^{u_B}\ar@{=}[u]&C_{u_B},\ar@{.>}[u]\ar[l]^{e_B}
}\]
and there exist two broken arrows which are KK-equivalences making the diagram commutative.
\end{thm}
Note that $D(-)$ is not a functor (i.e., everything is determined only up to KK-equivalence),
and there should be no way to determine the duality classes and dual sequences canonically in general (see Lem. \ref{b}).
In the proof of the above two theorems,
the K-groups must be finitely generated,
and the following proposition is used frequently.
\begin{prop}[{Prop. \ref{kmc}}]
Let $G$ be a finitely generated Abel group, and let $g_1, g_2\in G$ be two elements satisfying $G/\langle g_1\rangle \cong G/\langle g_2 \rangle$.
Then, there exists an automorphism $$G\ni g_1\mapsto g_2\in  G.$$
\end{prop}
\begin{rem}
For two reciprocal algebras $A$ and $B$,
$[1_A]_0\in \operatorname{Tor}(K_0(A))$ is equivalent to $[1_B]_0\not \in \operatorname{Tor}(K_0(B))$ (see Cor. \ref{exb}).
The class of unital UCT Kirchberg algebras with finitely generated K-groups is divided into two sub-classes,
where one is the class of C*-algebra $A$ with $[1_A]_0\in \operatorname{Tor}(K_0(A))$ and the other is of $B$ with $[1_B]_0\not \in \operatorname{Tor}(K_0(B))$,
and the reciprocality gives a one to one correspondence between these sub-classes.
\end{rem}

The categorical picture of the reciprocal algebras can be partially generalized to the setting of $C(X)$-algebras.
The unital map $\mathbb{C}\xrightarrow{u_A} A$ is replaced by unital $C(X)$-linear map $C(X)\xrightarrow{u_\mathcal{A}} \mathcal{A}$ of the locally trivial continuous $C(X)$-algebra with fiber $A$.
If we have the Spanier--Whitehead duality for $C(X)$-algebras and $KK_X$-groups, applying the duality to  $C_{u_\mathcal{A}}\xrightarrow{e_\mathcal{A}} C(X)$ gives another unital continuous field $\mathcal{B}$ with a unital map $\mathcal{B}\xleftarrow{u_\mathcal{B}} C(X)$.

As the first main result,
we prove the duality for locally trivial continuous $C(X)$-algebras.
We introduce the terminology $K_X$-dual to clarify the space $X$ on which the duality is discussed (see Def. \ref{SK}).
\begin{thm}[{Cor. \ref{CXd}}]\label{sm3}
Let $X$ be a finite CW-complex and let $C$ be a separable nuclear UCT C*-algebra with finitely generated K-groups.
For every locally trivial continuous $C(X)$-algebras $\mathcal{C}$ with fiber $C$,
there exists another locally trivial continuous $C(X)$-algebras $\mathcal{D}(\mathcal{C})$ with fiber $D(C)$ such that $\mathcal{C}$ and $\mathcal{D}(\mathcal{C})$ are Spanier--Whitehead $K_X$-dual.

In particular,
if $C$ and $D(C)$ are both stable Kirchberg algebras,
the duality gives a natural bijection
\[[X, \operatorname{BAut}(C)]\to [X, \operatorname{BAut}(D(C))].\]
\end{thm}
\begin{rem}
The $K_X$-duality is a duality for $KK_X$-groups,
and,
for example,
the equation $KK_X(\mathcal{C}, C(X))\cong KK_X(C(X), \mathcal{D}(\mathcal{C}))$ holds in contrast to the known K-duality $KK(C, \mathbb{C})\cong KK(\mathbb{C}, D(C))$.
The proof is done by cell-wisely untwisting the obstruction to extending duality classes.

The proof of the above theorem also provides a new proof of the existence of inverse for arbitrary element in the Dadarlat--Pennig cohomology group $E^1_{\mathcal{O}_\infty}(X)$.
The algebra $\mathcal{O}_\infty\otimes\mathbb{K}\sim_{KK}\mathbb{C}$ is known to be self-dual and the above bijection of the homotopy sets is exactly the map
\[E_{\mathcal{O}_\infty}^1(X)=[X, \operatorname{BAut}(\mathcal{O}_\infty\otimes\mathbb{K})]\ni c\mapsto -c\in E^1_{\mathcal{O}_\infty}(X).\]
\end{rem}

In the second main result,
we generalize the interesting bijection between the bundles of $\mathcal{O}_n$ and $\mathbb{M}_{n-1}(\mathcal{O}_\infty)$ via the categorical picture of reciprocal algebras and Thm. \ref{sm3}.
\begin{thm}[{Thm. \ref{taio}}]\label{RB}
Let $X$ be a finite CW-complex.
For the reciprocal Kirchberg algebras $A$ and $B$,
there is a natural bijection
\[R_{A, B} : [X, \operatorname{BAut}(A)]\to [X, \operatorname{BAut}(B)].\]
\end{thm}
\begin{rem}\label{Q}
Since the homotopy group $\pi_0(\operatorname{Aut}(A))$ is non-trivial, non-commutative in general and there are no known maps between the classifying spaces,
this bijection is not obvious even if one knows Thm. \ref{sm}.
For the same reason,
Thm. \ref{RB} does not imply Thm. \ref{sm},
and we do not know whether a similar statement is true or not for the based homotopy sets $[X, \operatorname{BAut}(A)]_0, [X, \operatorname{BAut}(B)]_0$.

For a given locally trivial $C(X)$-algebra $\mathcal{A}$ with fiber $A$, 
the diagram
\[
\xymatrix{
\mathcal{D}(C_{u_\mathcal{A}})&\mathcal{D}(C(X))\ar[l]\ar@{=}[d]&\mathcal{D}(\mathcal{A})\ar[l]\\
\mathcal{B}\ar@{.>}[u]&C(X)\ar[l]^{u_\mathcal{B}}&C_{u_\mathcal{B}}\ar[l]^{e_\mathcal{B}}
}\]
determines the locally trivial continuous $C(X)$-algebra $\mathcal{B}$ with fiber $B$,
and the bijection
\[R_{A, B} : [X, \operatorname{BAut}(A)]\ni [\mathcal{A}]\mapsto [\mathcal{B}]\in[X, \operatorname{BAut}(B)]\]
is defined by this procedure.
As in Thm. \ref{mn},
the broken arrow is a $KK_X$-equivalence.
However, the proof of bijectivity is still complicated because we do not even know whether the equivalence $C_{u_\mathcal{B}}\sim_{KK_X}\mathcal{D}(\mathcal{A})$ is true or not.
The key technical ingredient is J. Gabe's existence theorem \cite[Thm. E]{G}.
\end{rem}
{\bf Organization}
This paper consists of two parts.
In the first part from Sec. 2 to 3,
we give some preliminaries on the Spanier--Whitehead duality and show Thm. \ref{sm}, \ref{ca}.
In the second part from Sec. 4 to 5,
we discuss our main results on bundles of C*-algebras  and prove Thm. \ref{sm3}, \ref{RB}.
In appendix,
we give some computations on finitely generated Abel groups which are crucial for the reciprocality.
\[\]

{\bf Acknowledgments}
The author's deepest appreciation goes to Kan Kitamura who informed him of the relationship between the Spanier--Whitehead K-duality for $C(X)$-algebras and Brown's representation theory for the category $KK_X$ and gave him many interesting ideas including an argument computing cardinality of $KK_X$-groups.
The author is deeply grateful to Marius Dadarlat for many insightful comments and suggestions.
The author is greatly indebted to Masaki Izumi for his support and encouragement,
and he also would like to thank  Kengo Matsumoto for drawing the author's attention to \cite{KP}.
The author is supported by Research Fellow of the Japan Society for the Promotion of Science.  

\section{Preliminaries}
\subsection{Notation}
Let $A$ be a unital C*-algebra with the unit $1_A$ and map $u_A : \mathbb{C}\ni \lambda\mapsto \lambda 1_A\in A$.
Let $SA$ denote the suspension $C_0(0, 1)\otimes A$, and we write $S^iA:=S^{\otimes i}\otimes A$ for short.
Let $\mathbb{K}$ denote  the algebra of compact operators on the separable infinite dimensional Hilbert space,
and we denote by $\mathbb{M}_n$ the $n$ by $n$ matrix algebra.
For a $*$-homomorphism $\varphi : A\to B$,
we denote by $C_\varphi$ the mapping cone $\{(f, a)\in (C_0(0, 1]\otimes B)\oplus A \;|\; \varphi (a)=f(1)\}$,
and the following sequence is called the Puppe sequence (see \cite[Sec. 19]{B})
$$SA\xrightarrow{{\rm id}_S\otimes \varphi}SB\to C_\varphi \to A\xrightarrow{\varphi} B.$$
We denote by $KK(\varphi)\in KK(A, B)$ the Kasparov module represented by the $*$-homomorphism,
and the Kasparov product of two elements $KK(\varphi)\in KK(A, B)$, $KK(\psi)\in KK(B, C)$ is denoted by $KK(\varphi)\hat{\otimes}KK(\psi)=KK(\psi\circ\varphi)$ (see \cite[Sec. 18]{B}).
We write $I_A:=KK({\rm id}_A)\in KK(A, A)$.

In this paper,
we assume that the space $X$ is compact metrizable and of finite covering dimension, for example, a finite CW complex.
Let $C(X)$ be the C*-algebra of continuous functions on $X$,
and let $C_0(X, Y)$ be the ideal of functions vanishing on the closed subset $Y\subset X$.
For a unital $C(X)$-algebra $\mathcal{A}$,
we denote by $u_\mathcal{A} : C(X)\to \mathcal{A}$ the unital $*$-homomorphism which defines  the $C(X)$-linear structure of the C*-algebra.
Since $C_0(X, Y)\mathcal{A}=\{u_{\mathcal{A}}(f)a\in \mathcal{A}\; |\; f\in C_0(X, Y),\; a\in \mathcal{A}\}$ is a closed ideal of $\mathcal{A}$ by Cohen's factorization (see \cite[Th. 4.6.4]{Oz}), the algebra $\mathcal{A}(Y):=\mathcal{A}/C_0(X, Y)\mathcal{A}$ is a $C(Y)$-algebra with the quotient map $\pi_Y : \mathcal{A}\to\mathcal{A}(Y)$.
For $x\in X$,
we write $A_x:=A(\{x\}),\; \pi_x:=\pi_{\{x\}}$.
We write $KK_X(\mathcal{A}, \mathcal{B}):=\mathcal{R}KK(X: \mathcal{A}, \mathcal{B})$ for short,
where $\mathcal{R}KK(X: -, -)$ is the parametrized $KK$-group for $C(X)$-algebras introduced in \cite{K} (see also \cite{D3, MN}).

Let $[X, Y]$ be the set of homotopy classes of continuous maps from $X$ to $Y$.
and let $[X, Y]_0$ be the homotopy set of the maps between the pointed spaces concerning the base point preserving homotopy.
The $i$-th homotopy group of a space $Y$ is denoted by $\pi_i(Y):=[S^i, Y]_0$, where $S^i$ denotes the $i$-dimensional sphere.
For a unital C*-algebra $A$,
we denote by $\operatorname{Aut}(A)$ (resp. $\operatorname{End}(A)$) the set of automorphisms (resp. unital endomorphisms) equipped with the point-norm topology,
and the homotopy sets $[X, \operatorname{Aut}(A)]$ and $[X, \operatorname{End}(A)]$ have the natural semi-group structures.


\subsection{Bundles of C*-algebras and Dadarlat--Pennig theory}
A $C(X)$-algebra such that the map 
$$X\ni x\mapsto ||\pi_x(a)||_{\mathcal{A}_x}\in \mathbb{R},\quad a\in \mathcal{A}$$
is continuous is called continuous $C(X)$-algebra.
The continuous $C(X)$-algebra $\mathcal{A}$ is called locally trivial if there exists a closed neighborhood $Y$ for every $x\in X$ and a $C(Y)$-linear isomorphism $\mathcal{A}(Y)\cong C(Y)\otimes A$ for a C*-algebra $A$.
We always assume that $A$ is separable nuclear, and one can take the tensor product $\mathcal{A}\otimes_{C(X)}\mathcal{B}$ for two nuclear continuous $C(X)$-algebras (see \cite{Bla}).
One can identify locally trivial continuous $C(X)$-algebras with fiber $A$ and locally trivial principal $\operatorname{Aut}(A)$-bundles by the following elementary fact.
\begin{prop}
Let $X$ be a compact metrizable space, and let $A$ be a C*-algebra.
For every locally trivial continuous $C(X)$-algebra $\mathcal{A}$ with fiber $A$,
there exists a principal $\operatorname{Aut}(A)$ bundle $\mathcal{P}\to X$ such that the section algebra of the associated bundle, denoted by $\Gamma(X, \mathcal{P}\times_{\operatorname{Aut}(A)}A)$, is $C(X)$-linearly isomorphic to $\mathcal{A}$.
Using the above correspondence,
the set of isomorphism classes of principal $\operatorname{Aut}(A)$-bundles over $X$ is identified with the set of $C(X)$-linear isomorphism classes of locally trivial continuous $C(X)$-algebras with fiber $A$.
\end{prop}
Recall that the isomorphism class of a principal $\operatorname{Aut}(A)$-bundle $\mathcal{P}\to X$ is determined  by the homotopy class of the classifying map $f_{\mathcal{P}}$ making the following diagram commute :
$$\xymatrix{
\mathcal{P}\ar[r]\ar[d]&\operatorname{EAut}(A)\ar[d]\\
X\ar[r]^{f_{\mathcal{P}}\quad}&\operatorname{BAut}(A).}$$
Therefore, the set of isomorphism classes of locally trivial continuous $C(X)$-algebras with fiber $A$ is equal to $[X, \operatorname{BAut}(A)]$.
Thus, we denote by $[\mathcal{A}]\in [X, \operatorname{BAut}(A)]$ the isomorphism class of a locally trivial continuous $C(X)$-algebra $\mathcal{A}$ with fiber $A$.

In the case of $A$ is a stabilized strongly self-absorbing C*-algebra introduced in \cite{TW},
M. Dadarlat and U. Pennig reveal a surprising structure of $[X, \operatorname{BAut}(A)]$.

\begin{thm}[{\cite[{Thm. 3.8, Cor. 4.5}]{DP}}]\label{hn}
For every strongly self-absorbing C*-algebra $D$,
the group $\operatorname{Aut}(D\otimes \mathbb{K})$ is an infinite loop space providing a generalized cohomology $E^*_D$.
In particular, the homotopy set $$E^1_D(X) =[X, \operatorname{BAut}(D\otimes\mathbb{K})]$$ has a commutative group structure defined by the tensor product of locally trivial continuous $C(X)$-algebras with fiber $D\otimes\mathbb{K}$.
\end{thm}

The infinite Cuntz algebra $\mathcal{O}_\infty$ is a typical strongly self-absorbing Kirchberg algebra with the following K-groups :
$$(K_0(\mathcal{O}_\infty),\; [1_{\mathcal{O}_\infty}]_0,\; K_1(\mathcal{O}_\infty))\cong(\mathbb{Z}, 1, 0).$$
For every locally trivial continuous $C(X)$-algebra $\mathcal{A}$ with fiber $\mathcal{O}_\infty\otimes\mathbb{K}$,
there exists another one $\mathcal{A}^{-1}$ such that we have a $C(X)$-linear isomorphism $C(X)\otimes\mathcal{O}_\infty\otimes\mathbb{K}\cong \mathcal{A}\otimes_{C(X)}\mathcal{A}^{-1}$ (i.e., $[\mathcal{A}]+[\mathcal{A}^{-1}]=0\in E_{\mathcal{O}_\infty}^1(X)$).

\subsection{$KK_X$-groups}
We review some of the basic facts on $KK_X$-groups.
All C*-algebras in this paper are ungraded.
As in the usual $KK$-theory,
one has the natural maps
$$-\otimes I_\mathcal{C} : KK_X(\mathcal{A}, \mathcal{B})\ni KK(\phi)\mapsto KK(\phi\otimes {\rm id}_\mathcal{C})\in KK_X(\mathcal{A}\otimes_{C(X)}\mathcal{C}, \mathcal{B}\otimes_{C(X)} \mathcal{C}),$$
$$I_\mathcal{C}\otimes- : KK_X(\mathcal{A}, \mathcal{B})\ni KK(\phi)\mapsto KK({\rm id}_\mathcal{C}\otimes \phi)\in KK_X(\mathcal{C}\otimes_{C(X)}\mathcal{A}, \mathcal{C}\otimes_{C(X)} \mathcal{B}),$$
and the Kasparov product is denoted by
$$-\hat{\otimes}- : KK_X(\mathcal{A}, \mathcal{B})\times KK_X(\mathcal{B}, \mathcal{C})\to KK_X(\mathcal{A}, \mathcal{C}).$$
For the isomorphisms $\sigma : \mathcal{A}\otimes_{C(X)}\mathcal{C}\to \mathcal{C}\otimes_{C(X)} \mathcal{A}$ and $\theta : \mathcal{B}\otimes_{C(X)}\mathcal{C}\to \mathcal{C}\otimes_{C(X)}\mathcal{B}$ exchanging two tensor factors,
one has $\sigma\hat{\otimes}(I_\mathcal{C}\otimes -)\hat{\otimes}\theta^{-1}=-\otimes I_\mathcal{C}$.
In particular, the following equations hold for $a\in KK_X(\mathcal{A}, C(X))$ and $b\in KK_X(C(X), \mathcal{B})$ :
$$\sigma\hat{\otimes}(I_\mathcal{C}\otimes a)=a\otimes I_\mathcal{C}\in KK_X(\mathcal{A}\otimes_{C(X)}\mathcal{C},\mathcal{C}),\quad (I_\mathcal{C}\otimes b)\hat{\otimes}\theta^{-1}=b\otimes I_\mathcal{C}\in KK_X(\mathcal{C}, \mathcal{B}\otimes_{C(X)}\mathcal{C}).$$
In this paper, the Kasparov products are computed categorically, and the following commutativity is most important.
For $\alpha \in KK_X(\mathcal{A}, \mathcal{B}), \; \gamma \in KK_X(\mathcal{C}, \mathcal{D})$, one has
$$(\alpha\otimes I_\mathcal{C})\hat{\otimes}(I_\mathcal{B}\otimes \gamma)=(I_\mathcal{A}\otimes \gamma)\hat{\otimes}(\alpha \otimes I_\mathcal{D})\in KK_X(\mathcal{A}\otimes_{C(X)}\mathcal{C}, \mathcal{B}\otimes_{C(X)}\mathcal{D})$$
by \cite[Thm. 2.14. 8), Prop. 2.21.]{K} which is obvious if $\alpha, \gamma$ are given by the $*$-homomorphisms.

For the Puppe sequence
$$S\mathcal{A}\xrightarrow{I_S\otimes \varphi}S\mathcal{B}\xrightarrow{\iota} C_\varphi\xrightarrow{e} \mathcal{A}\xrightarrow{\varphi} \mathcal{B},$$
one has the following exact sequences (see \cite[Sec. 2]{MN})
$$KK_X(\mathcal{C}, S\mathcal{A})\to KK_X(\mathcal{C}, S\mathcal{B})\to KK_X(\mathcal{C}, C_\varphi)\to KK_X(\mathcal{C}, \mathcal{A})\to KK_X(\mathcal{C}, \mathcal{B}),$$
$$KK_X(\mathcal{B}, \mathcal{C})\to KK_X(\mathcal{A}, \mathcal{C})\to KK_X(C_\varphi, \mathcal{C})\to KK_X(S\mathcal{B}, \mathcal{C})\to KK_X(S\mathcal{A}, \mathcal{C}).$$
In \cite{MN},
the Puppe sequences and their shifts, for example $S\mathcal{A}\xrightarrow{S\varphi}S\mathcal{B}\xrightarrow{\iota}\mathcal{C}_\varphi\xrightarrow{e} \mathcal{A}$,
are proved to give the exact triangles of the triangulated category $KK_X$ which is the category of $C(X)$-algebras having $KK_X(-, -)$ as its morphisms.
We frequently use the following lemma in section \ref{5}.
\begin{lem}[{\cite[Sec. 2.1, Appendix]{MN}}]\label{RMN}
For two exact triangles $S\mathcal{B}_i\to\mathcal{C}_i\to \mathcal{A}_i\to\mathcal{B}_i, \;i=1, 2$ with the commutative diagram
\[\xymatrix{
S\mathcal{B}_1\ar[r]\ar[d]^{I_S\otimes\beta}&C_1\ar[r]&\mathcal{A}_1\ar[r]\ar[d]^{\alpha}&\mathcal{B}_1\ar[d]^{\beta}\\
S\mathcal{B}_2\ar[r]&\mathcal{C}_2\ar[r]&\mathcal{A}_2\ar[r]&\mathcal{B}_2,\\
}\]
there exists $\gamma\in KK_X(\mathcal{C}_1, \mathcal{C}_2)$ making the diagram commute.
Furthermore,
if $\alpha, \beta$ are $KK_X$-equivalences,
then $\gamma$ is also a $KK_X$-equivalence.
\end{lem}

The definition of the $KK_X$-equivalence is the same as in the usual $KK$-theory and we denote by $\sim_{KK_X}$ the equivalence relation.
We denote by $KK_X(\mathcal{A}, \mathcal{B})^{-1}$ the set of $KK_X$-equivalences.
Since $KK_X$ is contravariant with respect to $X$ (see \cite[Prop. 2.20.]{K}),
one has the evaluation map 
$$KK_X(\mathcal{A}, \mathcal{B})\ni\sigma\mapsto \sigma_x\in KK(\mathcal{A}_x, \mathcal{B}_x), \quad x\in X,$$
and M. Dadarlat characterizes the $KK_X$-equivalence using this map.
\begin{thm}[{\cite[Thm. 1.1, Thm. 2.7.]{D3}}]\label{Ddeq}
Let $\mathcal{A}, \mathcal{B}$ be separable nuclear continuous $C(X)$-algebras,
where $X$ is a compact metrizable space of finite covering dimension.
Then, $\sigma\in KK_X(\mathcal{A}, \mathcal{B})$ is a $KK_X$-equivalence if and only if $\sigma_x\in KK(\mathcal{A}_x, \mathcal{B}_x)^{-1}$ for every $x\in X$.

If the fibers $\mathcal{A}_x, \mathcal{B}_x$ are Kirchberg algebras, there is a $C(X)$-linear isomorphisms $$\phi : \mathcal{A}\otimes\mathbb{K}\to \mathcal{B}\otimes\mathbb{K}$$
with $\sigma =KK(\phi)$.
\end{thm}
Here, for the lifting of $\sigma \in KK_X(\mathcal{A}, \mathcal{B})$ to the $C(X)$-linear $*$-homomorphism $\phi$,
we need the existence theorem known as Kirchberg--Gabe's theorem.
\begin{thm}[{\cite{Ki}, \cite[Thm. E, F]{G}}]\label{KGE}
Let $X$ be a compact metrizable space,
and let $\mathcal{A}$ and $\mathcal{B}$ be locally trivial continuous $C(X)$-algebras whose fibers are Kirchberg algebras.
Then,
every $\sigma\in KK_X(\mathcal{A}, \mathcal{B})$ is represented by a $C(X)$-linear $*$-homomorphism $\phi : \mathcal{A}\to\mathcal{B}$ (i.e., $KK_X(\phi)=\sigma$).

If $\mathcal{B}$ is stable and two homomorphisms $\phi_0, \phi_1$ satisfy $KK_X(\phi_0)=KK_X(\phi_1)$,
then there is a continuous path of $C(X)$-linear $*$-homomorphisms $\{\phi_t\}_{t\in [0, 1]}$ connecting $\phi_0$ and $\phi_1$.
\end{thm}
These $C(X)$-algebras in the above theorem are $\mathcal{O}_\infty$-stable, tight $X$-algebras in \cite{G},
and the above statement is a special case of \cite[Thm. E]{G}.
Since $B$ is stable, $\phi_0$ and $\phi_1$ are asymptotically unitary equivalent via multiplier unitaries.
The unitary group $U(\mathcal{M}(\mathcal{B}))$ is contractible and this makes $\phi_0$ and $\phi_1$ homotopy equivalent.

For a given separable nuclear C*-algebra $A$,
there is a unital Kirchberg algebra $KK$-equivalent to $A$. 
By \cite[Proof of Thm. 2.5, Proof of Lem. 2.2.]{D3},
one can easily check a similar statement for locally trivial continuous $C(X)$-algebras.
\begin{thm}[{\cite[Thm. 2.5.]{D3}}]\label{de}
Let $X$ be a compact metrizable space of finite covering dimension.
Let $\mathcal{A}$ be a locally trivial continuous $C(X)$-algebra with fiber a separable nuclear C*-algebra $A$.
Then, there exists a unital locally trivial continuous $C(X)$-algebra $\mathcal{A}^\sharp$ satisfying $\mathcal{A}\sim_{KK_X} \mathcal{A}^\sharp$,
and the fiber ${\mathcal{A}^\sharp}_x(\sim_{KK}A)$ is a unital Kirchberg algebra.
\end{thm}

\subsection{Spanier--Whitehead duality}
We introduce the Spanier--Whitehead duality for $KK_X$-groups following the formulation in \cite{KS}.
\begin{dfn}[{cf. \cite[Def. 2.1.]{KS}}]\label{SK}
Two $C(X)$-algebras $\mathcal{A}$ and $\mathcal{B}$ are said to have duality classes if there exist
$$\mu \in KK_X(C(X), \mathcal{A}\otimes _{C(X)}\mathcal{B}),\quad \nu \in KK_X(\mathcal{B}\otimes _{C(X)}\mathcal{A}, C(X))$$
satisfying $$(\mu\otimes I_\mathcal{A})\hat{\otimes}(I_\mathcal{A}\otimes \nu)=I_\mathcal{A},$$
$$(I_\mathcal{B}\otimes \mu)\hat{\otimes}(\nu\otimes I_\mathcal{B})=I_\mathcal{B}.$$
Then,
$\mathcal{A}$ and $\mathcal{B}$ are said to be Spanier--Whitehead $K_X$-dual with duality classes $(\mu, \nu)$.
\end{dfn}
\begin{rem}\label{Sd}
The following diagrams might be helpful to follow the computations using the duality classes.
The morphisms $\mu, I_\mathcal{A}, \nu$ are represented by
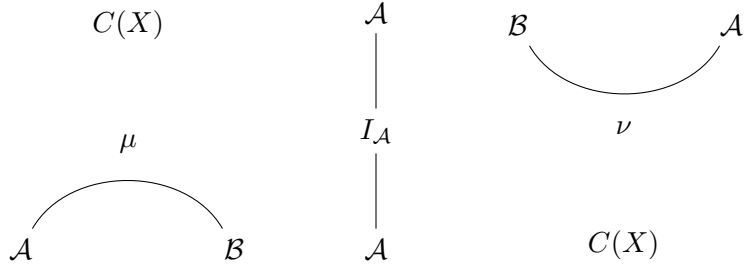
\[
\begin{tikzpicture}
\node (O) {$\mu$};
\node (C) [above=of O] {$C(X)$};
\node (OO) [below=of O] {};
\node (A) [left=of OO] {$\mathcal{A}$};
\node (D) [right=of OO] {$\mathcal{B}$};
\draw[-] (A) to[out=60,in=120] (D);
\end{tikzpicture}
\begin{tikzpicture}
\node (A2) {$\mathcal{A}$};
\node (O) [above=of A2]{$I_\mathcal{A}$};
\node (A1) [above=of O] {$\mathcal{A}$};
\node (O1) [left=of O] {};
\node (O2) [right=of O] {};
\draw (A1)--(O);
\draw (A2)--(O);
\end{tikzpicture}
\begin{tikzpicture}
\node (OO) {};
\node (B) [left=of OO] {$\mathcal{B}$};
\node (A) [right=of OO] {$\mathcal{A}$};
\node (O) [below=of OO] {$\nu$};
\node (C) [below=of O] {$C(X)$};
\draw (B) to[out=-60, in=-120] (A);
\end{tikzpicture}
\]
and the unit counit adjunction formula $(\mu\otimes I_\mathcal{A})\hat{\otimes}(I_\mathcal{A}\otimes\nu)=I_\mathcal{A}$ is represented by
\[
\begin{tikzpicture}
\node (B) {$\mathcal{B}$};
\node (A2) [left=of B] {$\mathcal{A}$};
\node (A3) [right=of B] {$\mathcal{A}$};
\node (A1) [below=of A2] {$\mathcal{A}$};
\node (A4) [above=of A3] {$\mathcal{A}$};
\draw (A1)--(A2);
\draw (A2) to[out=45, in=135] (B);
\draw (B) to[out=-45, in=-135] (A3);
\draw (A3)--(A4);
\end{tikzpicture}
\begin{tikzpicture}
\node (O) {=};
\node (On) [above=of O] {};
\node (Ow) [left=of O] {};
\node (Oe) [right=of O] {};
\node (Os) [below=of O] {};
\end{tikzpicture}
\begin{tikzpicture}
\node (A1) {$\mathcal{A}$};
\node (O) [below=of A1] {};
\node (A2) [below=of O] {$\mathcal{A}$.};
\draw (A1)--(A2);
\end{tikzpicture}
\]
\end{rem}

\begin{prop}[{\cite[Thm. 2.6.]{KS}}]
The definition of duality is symmetric for $\mathcal{A}$ and $\mathcal{B}$.
More precisely,
using the isomorphism 
$$\sigma : \mathcal{A}\otimes_{C(X)}\mathcal{B}\ni a\otimes b\mapsto b\otimes a\in \mathcal{B}\otimes_{C(X)}\mathcal{A},$$
$\mathcal{B}$ and $\mathcal{A}$ are also $K_X$-dual with duality classes $(\mu\hat{\otimes}\sigma, \sigma\hat{\otimes}\nu)$.
\end{prop}
\begin{proof}
Using the isomorphisms $\theta_{B, BA} : \mathcal{B}\otimes_{C(X)}(\mathcal{B}\otimes_{C(X)}\mathcal{A})\to (\mathcal{B}\otimes_{C(X)}\mathcal{A})\otimes_{C(X)}\mathcal{B}$ and $\theta_{B, AB} : \mathcal{B}\otimes_{C(X)}(\mathcal{A}\otimes_{C(X)}\mathcal{B})\to (\mathcal{A}\otimes_{C(X)}\mathcal{B})\otimes_{C(X)}\mathcal{B}$ exchanging tensor factors,
one can obtain the following equations $$(\mu\hat{\otimes}\sigma)\otimes I_\mathcal{B}=(I_\mathcal{B}\otimes(\mu\hat{\otimes}\sigma))\hat{\otimes}\theta_{B, BA},$$
$$I_\mathcal{B}\otimes (\sigma\hat{\otimes}\nu)=\theta_{B, AB}\hat{\otimes}((\sigma\hat{\otimes}\nu)\otimes I_\mathcal{B}).$$
It is straightforward to check $$((\mu\hat{\otimes}\sigma)\otimes I_\mathcal{B})\hat{\otimes}(I_\mathcal{B}\otimes(\sigma\hat{\otimes}\nu))=(I_\mathcal{B}\otimes\mu)\hat{\otimes}(\nu\otimes I_\mathcal{B})=I_\mathcal{B},$$
and the other equation for the duality classes is proved similarly.
\end{proof}
The following results which are obtained by a similar argument as in \cite[Proof of Thm. 2.2.]{KS} are known for the specialists,
but we write their proofs for the convenience of readers.
\begin{lem}\label{b}
Let $\mathcal{A}$ and $\mathcal{B}$ be $C(X)$-algebras with two elements
$$\mu\in KK_X(C(X), \mathcal{A}\otimes_{C(X)}\mathcal{B}),\quad \nu\in KK_X(\mathcal{B}\otimes_{C(X)}\mathcal{A}, C(X)).$$

\end{lem}
\begin{proof}
First, we prove 1.
Let $\mu \hat{\otimes}(\alpha^{-1}\otimes I_\mathcal{B})$ denote by $\mu_\alpha$,
then one has
\begin{align*}
I_\mathcal{A}=&(\mu\otimes I_\mathcal{A})\hat{\otimes}(I_\mathcal{A}\otimes \nu)\hat{\otimes}\alpha^{-1}\\
=&(\mu\otimes I_\mathcal{A})\hat{\otimes}(I_\mathcal{A}\otimes \nu)\hat{\otimes}(\alpha^{-1}\otimes I_{C(X)})\\
=&(\mu\otimes I_\mathcal{A})\hat{\otimes}(\alpha^{-1}\otimes I_{\mathcal{B}\otimes_{C(X)}\mathcal{A}})\hat{\otimes}(I_\mathcal{A}\otimes\nu)\\
=&(\mu_\alpha\otimes I_\mathcal{A})\hat{\otimes}(I_\mathcal{A}\otimes\nu),
\end{align*}
and direct computation yields
\begin{align*}
\beta\hat{\otimes}(I_\mathcal{B}\otimes\mu_\alpha)\hat{\otimes}(\nu\otimes I_\mathcal{B})=&(\beta\otimes I_{C(X)})\hat{\otimes}(I_\mathcal{B}\otimes\mu_\alpha)\hat{\otimes}(\nu\otimes I_\mathcal{B})\\
=&(I_\mathcal{B}\otimes\mu_\alpha)\hat{\otimes}(\beta\otimes I_{\mathcal{A}\otimes_{C(X)}\mathcal{B}})\hat{\otimes}(\nu\otimes I_\mathcal{B})\\
=&(I_\mathcal{B}\otimes \mu)\hat{\otimes}(I_\mathcal{B}\otimes \alpha^{-1}\otimes I_\mathcal{B})\hat{\otimes}(((\beta\otimes I_\mathcal{A})\hat{\otimes}\nu)\otimes I_\mathcal{B}).
\end{align*}
We also have
\begin{align*}
(\beta\otimes I_\mathcal{A})\hat{\otimes}\nu=&(I_\mathcal{B}\otimes \mu\otimes I_\mathcal{A})\hat{\otimes}(\nu\otimes I_{\mathcal{B}\otimes_{C(X)}\mathcal{A}})\hat{\otimes}(I_{C(X)}\otimes \nu)\\
=&(I_\mathcal{B}\otimes \mu\otimes I_\mathcal{A})\hat{\otimes}(I_{\mathcal{B}\otimes_{C(X)}\mathcal{A}}\otimes \nu)\hat{\otimes}(\nu\otimes I_{C(X)})\\
=&(I_\mathcal{B}\otimes\alpha)\hat{\otimes}\nu,
\end{align*}
and this yields
$$\beta\hat{\otimes}(I_\mathcal{B}\otimes\mu_\alpha)\hat{\otimes}(\nu\otimes I_\mathcal{B})=\beta$$
which proves 1.

The proof of 2 is done by a similar computation.
So we only show $\alpha^{-1}=\alpha'$,
and it is enough to check $\alpha\hat{\otimes}\alpha'=I_\mathcal{A}$ by symmetry.
The direct computation yields
\begin{align*}
\alpha\hat{\otimes}\alpha'=&(I_{C(X)}\otimes\mu\otimes I_\mathcal{A})\hat{\otimes}(I_{C(X)}\otimes I_\mathcal{A}\otimes \nu')\hat{\otimes}(((\mu'\otimes I_\mathcal{A})\hat{\otimes}(I_\mathcal{A}\otimes \nu)) \otimes I_{C(X)})\\
=&(I_{C(X)}\otimes \mu \otimes I_\mathcal{A})\hat{\otimes}(((\mu'\otimes I_\mathcal{A})\hat{\otimes}(I_\mathcal{A}\otimes \nu))\otimes I_{\mathcal{B}\otimes_{C(X)}\mathcal{A}})\hat{\otimes}(I_\mathcal{A}\otimes I_{C(X)}\otimes\nu')\\
=&(\mu'\otimes I_{C(X)}\otimes I_\mathcal{A})\hat{\otimes}(I_\mathcal{A}\otimes I_\mathcal{B}\otimes\mu\otimes I_\mathcal{A})\hat{\otimes}(I_\mathcal{A}\otimes\nu\otimes I_\mathcal{B}\otimes I_\mathcal{A})\hat{\otimes}(I_\mathcal{A}\otimes I_{C(X)}\otimes \nu')\\
=&(\mu'\otimes I_\mathcal{A})\hat{\otimes}(I_\mathcal{A}\otimes \nu')\\
=&I_\mathcal{A}.
\end{align*}
Statement 3 is shown by the same computation,
and 4 follows by definition.
\end{proof}
Among the 4 statements on the duality classes, 
statement 1. is crucial in Sec. \ref{spwh}.  
From the above lemma and its proof,
one can easily show the following.
\begin{cor}[{\cite[Cor. 2.3.]{KS}}]\label{pd}
Let $\mathcal{A}$ and $\mathcal{B}$ be $K_X$-dual with the duality classes $(\mu, \nu)$.
Then,
the map $$\mu\hat{\otimes}(-\otimes I_\mathcal{B}) : KK_X(\mathcal{A}, C(X))\to KK_X(C(X), \mathcal{B})$$
is an isomorphism with the inverse map $(-\otimes I_\mathcal{A})\hat{\otimes}\nu$.
\end{cor}
\begin{cor}[{cf. \cite[Sec. 3.3]{P}}]\label{up}
Let $\mathcal{A}$ be a locally trivial continuous $C(X)$-algebra with fiber $\mathcal{O}_\infty\otimes\mathbb{K}$,
and let $\mathcal{A}^{-1}$ be its inverse in $E^1_{\mathcal{O}_\infty}(X)$ (i.e., $\mathcal{A}\otimes_{C(X)}\mathcal{A}^{-1}\cong C(X)\otimes \mathcal{O}_\infty\otimes\mathbb{K}$).
Then, $\mathcal{A}$ and $\mathcal{A}^{-1}$ are  $K_X$-dual.
\end{cor}
We will discuss the $K_X$-duality for $C(X)$-algebras in Sec. 4,
and we restrict ourselves to the case $X=pt$ in the remainder of this section and Sec. 3.
Not all $C(X)$-algebras admit the $K_X$-duality,
but the following fact is known in the case $X=pt$.
\begin{thm}[{\cite[Thm. 3.1, 6.2]{KS}, \cite[Sec. 4.4.]{KPW}}]\label{SWK}
Let $A$ be an arbitrary separable nuclear UCT C*-algebra.
One has a separable nuclear UCT C*-algebra $D(A)$ such that $A$ and $D(A)$ are Spanier--Whitehead K-dual if and only if $A$ has finitely generated K-groups. 
\end{thm}
\begin{rem}\label{usi}
By Cor. \ref{pd} and the UCT,
one can determine $D(A)$ from the K-groups of $A$.
For example,
$\mathbb{C}$ and $\mathbb{S}$ are self-dual,
and $\mathcal{O}_n$ and $S\mathcal{O}_n$ are K-dual.
We note that the duality class $\mu_S\in KK(\mathbb{C}, S^2)$ can be chosen as the Bott element.

Every UCT C*-algebras with finitely generated K-groups are $KK$-equivalent to the direct sum of the above building blocks,
and one can check the equivalence $D(A)\oplus D(B)\sim_{KK} D(A\oplus B)$ which enables us to describe the K-groups of $A$ and $D(A)$ completely.
\end{rem}
\begin{lem}[{\cite[Thm. 2.2, Cor. 2.3.]{KS}}]\label{b1}
Let $A$ (resp. $B$ and $C$) be C*-algebra with the duality classes 
$$\mu_A\in KK(\mathbb{C}, A\otimes D(A)),\quad \nu_A\in KK(D(A)\otimes A, \mathbb{C}).$$
The map $d^X_{\mu_A, \nu_B} : KK(A, B\otimes C(X))\to KK(D(B), C(X)\otimes D(A))$ defined by $$d^X_{\mu_A, \nu_B}(x):= (I_{D(B)}\otimes \mu_A)\hat{\otimes}(I_{D(B)}\otimes x\otimes I_{D(A)})\hat{\otimes}(\nu_B\otimes I_{C(X)\otimes D(A)})$$
is an isomorphism satisfying
$$d^{X\times Y}_{\mu_A, \nu_C}(y\hat{\otimes}(x\otimes I_{C(Y)}))=d^X_{\mu_B, \nu_C}(x)\hat{\otimes}(I_{C(X)}\otimes d^Y_{\mu_A, \nu_B}(y))$$
for $y\in KK(A, B\otimes C(Y))$ and $x\in KK(B, C\otimes C(X))$.

In particular,
the map $$d_{\mu_A, \nu_A}:=d^{pt}_{\mu_A, \nu_A} : KK(A, A)\to KK(D(A), D(A))$$ is an anti-isomorphism of the rings.
\end{lem}
\begin{proof}
By a similar computation as in Lem. \ref{b}, the inverse map of $d^X_{\mu_A, \nu_B}$ is given by 
$$KK(D(B), C(X)\otimes D(A))\ni z \mapsto (\mu_B \otimes I_A)\hat{\otimes}(I_B\otimes z \otimes I_A)\hat{\otimes}(I_{B\otimes C(X)}\otimes \nu_A)\in KK(A, B\otimes C(X)).$$

We show the equation
$$d^{X\times Y}_{\mu_A, \nu_C}(y\hat{\otimes}(x\otimes I_{C(Y)}))=d^X_{\mu_B, \nu_C}(x)\hat{\otimes}(I_{C(X)}\otimes d^Y_{\mu_A, \nu_B}(y)).$$
By $(\mu_B\otimes I_B)\hat{\otimes}(I_B\otimes \nu_B)=I_B$ and the commutativity of $\hat{\otimes}$,
we have
\begin{align*}
y\hat{\otimes}(x\otimes I_{C(Y)})=&y\hat{\otimes}(\mu_B\otimes I_B\otimes I_{C(Y)})\hat{\otimes}(I_B\otimes\nu\otimes I_{C(Y)})\hat{\otimes}(x\otimes I_{C(Y)})\\
=&(\mu_B\otimes I_A)\hat{\otimes}(I_{B\otimes D(B)}\otimes y)\hat{\otimes}(x\otimes I_{D(B)\otimes B\otimes C(Y)})\hat{\otimes}(I_{C\otimes C(X)}\otimes\nu_B\otimes I_{C(Y)})\\
=&(\mu_B\otimes I_A)\hat{\otimes}(x\otimes I_{D(B)\otimes A})\hat{\otimes}(I_{C\otimes C(X)\otimes D(B)}\otimes y)\hat{\otimes}(I_{C\otimes C(X)}\otimes \nu_B\otimes I_{C(Y)}).
\end{align*}
We write $Z:=I_{D(C)}\otimes ((\mu_B\otimes I_A)\hat{\otimes}(x\otimes I_{D(B)\otimes A}))\otimes I_{D(A)}$,
and direct computation yields
\begin{align*}
(I_{D(C)}\otimes \mu_A)\hat{\otimes}Z=&(I_{D(C)}\otimes \mu_A)\hat{\otimes}(I_{D(C)}\otimes \mu_B\otimes I_{A\otimes D(A)})\hat{\otimes}(I_{D(C)}\otimes x\otimes I_{D(B)\otimes A\otimes D(A)})\\
=&(I_{D(C)}\otimes \mu_B)\hat{\otimes}(I_{D(C)\otimes B\otimes D(B)}\otimes \mu_A)\hat{\otimes}(I_{D(C)}\otimes x\otimes I_{D(B)\otimes A\otimes D(A)})\\
=&(I_{D(C)}\otimes \mu_B)\hat{\otimes}(I_{D(C)}\otimes x\otimes I_{D(B)})\hat{\otimes}(I_{D(C)\otimes C\otimes C(X)\otimes D(B)}\otimes \mu_A).
\end{align*}
Similarly,
one can check
\begin{align*}
&(I_{D(C)\otimes C\otimes C(X)\otimes D(B)}\otimes y\otimes I_{D(A)})\hat{\otimes}(I_{D(C)\otimes C\otimes C(X)}\otimes \nu_B\otimes I_{C(Y)\otimes D(A)})\hat{\otimes}(\nu_C\otimes I_{C(X\times Y)\otimes D(A)})\\
=&(\nu_C\otimes I_{C(X)\otimes D(B)\otimes A\otimes D(A)})\hat{\otimes}(I_{C(X)}\otimes ((I_{D(B)}\otimes y\otimes I_{D(A)})\hat{\otimes}(\nu_B\otimes I_{C(Y)\otimes D(A)})))=:W,
\end{align*}
and the commutativity shows
\begin{align*}
&(I_{D(C)\otimes C\otimes C(X)\otimes D(B)}\otimes \mu_A)\hat{\otimes}(\nu_C\otimes I_{C(X)\otimes D(B)\otimes A\otimes D(A)})\\
=&(\nu_C\otimes I_{C(X)\otimes D(B)})\hat{\otimes}(I_{C(X)}\otimes (I_{D(B)}\otimes \mu_A)).
\end{align*}
Finally,
we get
\begin{align*}
d_{\mu_A, \nu_C}^{X\times Y}(y\hat{\otimes}(x\otimes I_{C(Y)}))=&(I_{D(C)}\otimes\mu_A)\hat{\otimes}Z\hat{\otimes}W\\
=&d^X_{\mu_B, \nu_C}(x)\hat{\otimes}(I_{C(X)}\otimes d^Y_{\mu_A, \nu_B}(y)).
\end{align*}
\end{proof}
The map $d^X_{\mu_A, \nu_B}$ is the composition of two maps
$$KK(A, B\otimes C(X))\ni x\mapsto (I_{D(B)}\otimes x)\hat{\otimes}(\nu_B\otimes I_{C(X)})\in KK(D(B)\otimes A, C(X)),$$
$$KK(D(B)\otimes A, C(X))\ni y \mapsto (I_{D(B)}\otimes \mu_A)\hat{\otimes}(y\otimes I_{D(A)})\in KK(D(B), C(X)\otimes D(A)).$$
So $d^X_{\mu_A, \nu_B}$ is natural with respect to $X$.

\begin{lem}\label{b2}
Let $A, B$ and $C$ be C*-algebras as in Lem. \ref{b1}.
Assume 
$$KK(C, S^i)\xrightarrow{h\hat{\otimes}} KK(B, S^i)\xrightarrow{f\hat{\otimes}} KK(A, S^i)\quad i=0, 1,$$ are exact for $f\in KK(A, B),\; h\in KK(B, C)$.
Then,
the following sequences are exact
$$K_i(D(C))\xrightarrow{\hat{\otimes}d_{\mu_B, \nu_C}(h)} K_i(D(B))\xrightarrow{\hat{\otimes}d_{\mu_A, \nu_B}(f)} K_i(D(A))\quad i=0, 1.$$
\end{lem}
\begin{proof}
We identify $K_i(-)$ with $KK(S^i, -)$ and choose duality classes $$\mu_{S^i}\in KK(\mathbb{C}, S^i\otimes S^i),\quad \nu_{S^i}\in KK(S^i\otimes S^i, \mathbb{C}).$$
By Lem. \ref{b1},
every element of $KK(S^i, D(C))$ is of the form $d_{\mu_C, \nu_{S^i}}(c), \; c\in KK(C, S^i)$,
and one has
$$d_{\mu_C, \nu_{S^i}}(c)\hat{\otimes}d_{\mu_B, \nu_{C}}(h)\hat{\otimes}d_{\mu_A, \nu_{B}}(f)=d_{\mu_A, \nu_{S^i}}(f\hat{\otimes} h\hat{\otimes}c)=0$$
by the assumption of $f$ and $h$.

Suppose that $b\in KK(B, S^i)$ satisfies $$0=d_{\mu_B, \nu_{S^i}}(b)\hat{\otimes}d_{\mu_A, \nu_{B}}(f)=d_{\mu_A, \nu_{S^i}}(f\hat{\otimes}b).$$
Since the map $d_{\mu_A, \nu_{S^i}}$ is injective,
one has $f\hat{\otimes}b=0$.
By assumption,
there is $c\in KK(C, S^i)$ satisfying $b=h\hat{\otimes} c$,
which implies $$d_{\mu_B, \nu_{S^i}}(b)=d_{\mu_C, \nu_{S^i}}(c)\hat{\otimes}d_{\mu_B, \nu_{C}}(h).$$
This proves the statement.
\end{proof}

Recall the mapping cone sequence (see \cite[Thm. 19.4.3.]{B})
$$\xymatrix{
SSA\ar[r]^{I_S\otimes \iota_A}&SC_{u_A}\ar[r]^{I_S\otimes e_A}& S\mathbb{C}\ar[r]^{I_S\otimes u_A}&SA\ar[r]^{\iota_A}&C_{u_A} \ar[r]^{e_A}&\mathbb{C}\ar[d]^{\mu_S}\ar[r]^{u_A}& A\ar[d]^{\mu_S\otimes I_A}\\
&&&&&SS\ar[r]^{I_{S^2}\otimes u_A}& SSA.
}$$
Applying Lem. \ref{b2} to 
$$S\to SA\to C_{u_A}\xrightarrow{e_A}\mathbb{C}\xrightarrow{u_A} A\xrightarrow{(\mu_S\otimes I_A)\hat{\otimes}(I_S\otimes \iota_A)}SC_{u_A}\to S,$$
one has an exact sequence
$$0\leftarrow K_1(D(A))\leftarrow K_0(D(C_{u_A}))\leftarrow \mathbb{Z}\leftarrow K_0(D(A))\leftarrow K_1(D(C_{u_A}))\leftarrow 0,$$
and there is a unital UCT Kirchberg algebra $B\sim_{KK}D(C_{u_A})$ making the following diagram commute (see \cite[Sec.4.3.]{Ro})
$$\xymatrix{
K_0(D(C_{u_A}))\ar@{=}[d]&\mathbb{Z}\ar[l]\ar@{=}[d]\\
K_0(B)&K_0(\mathbb{C}).\ar[l]^{(u_B)_*}
}$$
Then, the equations $K_i(C_{u_B})\cong K_i(D(A))$ automatically hold.
Since the rank of the free part of $K_i(A)$ is equal to that of $K_i(D(A))$,
one can observe $A\not\sim_{KK} B$ by comparing ranks of the free parts of their K-groups.
Furthermore, if $[1_A]_0\in K_0(A)$ is a torsion element, $[1_B]_0\in K_0(B)$ is not.
Conversely, if $[1_A]_0\in K_0(A)$ is not a torsion element,
then $[1_B]_0\in K_0(B)$ is a torsion element.
So one has the following corollary.
\begin{cor}\label{exb}
For every unital UCT Kirchberg algebra $A$ with finitely generated K-groups,
there exists a unital UCT Kirchberg algebra $B\not \sim_{KK} A$ satisfying
$$D(C_{u_A})\sim_{KK} B,\quad D(C_{u_B})\sim_{KK} A.$$
One has $[1_A]_0\in \operatorname{Tor}(K_0(A))$ if and only if $[1_B]_0\not\in \operatorname{Tor}(K_0(B))$.
\end{cor}
\subsection{$K_0(A)$ and $K_1(C_{u_A})$}\label{nota}
In this section,
we investigate a relationship between $K_0(A)$ and $K_1(C_{u_A})$.
Every finitely generated Abel group $G$ admits the following presentation
$$G=\mathbb{Z}^{ F}\oplus \bigoplus_{p\colon {\rm prime}}G(p),$$
where we write $G(p)=\mathbb{Z}_{p^{k_1}}\oplus\cdots\oplus\mathbb{Z}_{p^{k_{t_p}}},\; k_i>0$ for finitely many $p$ and otherwise $G(p)=0$.
We write $I(G(p)) : =\{k_1, \cdots, k_{t_p}\}$ if $G(p)\not =0$.
For the above $p$-group $G(p)$, we define $L(G(p))$ by $$L(G(p)):=t_p=|I(G(p))|\;\;\; {\rm for}\; G(p)\not =0, \quad L(0):=0.$$
For example,
one has $$I(\mathbb{Z}_p\oplus\mathbb{Z}_p)=\{1, 1\},\quad I(\mathbb{Z}_p\oplus\mathbb{Z}_p\oplus\mathbb{Z}_{p^2})=I(\mathbb{Z}_p\oplus\mathbb{Z}_{p^2}\oplus\mathbb{Z}_p)=\{1, 1, 2\},$$
$$L(\mathbb{Z}_p\oplus\mathbb{Z}_p\oplus\mathbb{Z}_{p^2})=3,\quad I(\mathbb{Z}_p\oplus \mathbb{Z}_p)\cap I(\mathbb{Z}_p\oplus\mathbb{Z}_p\oplus\mathbb{Z}_{p^2})=\{1,1\}.$$
We note the precise meaning of the intersection $I(G(p))\cap I(H(p))$.
It is the collection of numbers identified with the subset of both  $I(G(p))$ and $I(H(p))$ such that the complements $I(G(p))\backslash (I(G(p))\cap I(H(p)))$ and $I(H(p))\backslash (I(G(p))\cap I(H(p)))$ do not share the same number.
Since every $I(G(p))$ is a finite set,
the intersection is well-defined and uniquely determined by $I(G(p))$ and $I(H(p))$.
We also write $I(G(p))\backslash I(H(p)):= I(G(p))\backslash (I(G(p))\cap I(H(p)))$.

Let $(G, \tilde{G})$ be a pair of finite Abel $p$-groups (i.e., $G=G(p), \tilde{G}=\tilde{G}(p)$).
The pair $(G, \tilde{G})$ is said to satisfy ($*$), if either $G=\tilde{G}$, or one can take the strictly increasing sequence by alternating all elements of $I(G)\backslash I(\tilde{G})$ and $I(\tilde{G})\backslash I(G)$.
More precisely,
a pair $(G, \tilde{G})$ with $G\not =\tilde{G}$ satisfies ($*$) if and only if one has 
$$G=N\oplus \mathbb{Z}_{p^{k_1}}\oplus\cdots\oplus\mathbb{Z}_{p^{k_s}},\quad \tilde{G}=N\oplus\mathbb{Z}_{p^{\tilde{k}_1}}\oplus\cdots\oplus\mathbb{Z}_{p^{\tilde{k}_{s}}}$$ for some finite Abel p-group $N$ and $k_i, \tilde{k}_j$ that satisfy either $0\leq k_1<\tilde{k}_1<\cdots <k_s<\tilde{k}_{s}$, or $0\leq \tilde{k}_1<k_1<\cdots<\tilde{k}_{s} <k_s$ for $1\leq s$.

For a presentation $N=\mathbb{Z}_{p^{n_1}}\oplus\cdots\oplus \mathbb{Z}_{p^{n_t}}$, we have
$$I(G)\backslash I(\tilde{G})=\{k_i \;|\; 0<k_i\}, \quad I(\tilde{G})\backslash I(G)=\{\tilde{k}_j\; |\; 0<\tilde{k}_j\},$$
$$I(G)\cap I(\tilde{G})=\{n_1,\cdots, n_t\}.$$
In particular,
if $(G, \tilde{G})$ satisfies ($*$), then so does $(\tilde{G}, G)$,
and one has $|L(G)-L(\tilde{G})|\leq 1$.
For a pair $(G, \tilde{G})\not =(0, 0)$ with ($*$),
we write $m(G, \tilde{G}): ={\rm max}\;\{k\;|\; k\in I(G)\cup I(\tilde{G})\}$.

A pair $(G, \tilde{G})$ with ($*$) is said to satisfy ($**$) if either $G=\tilde{G}$ or one has 
$${\rm max}\;\{k \;|\; k\in (I(\tilde{G})\backslash I(G))\cup \{0\}\}>{\rm max}\;\{k \;|\; k\in (I(G)\backslash I(\tilde{G}))\cup \{0\}\}.$$


\begin{prop}\label{vn}
For a unital C*-algebra $A$ with finitely generated K-groups, the following hold.

\end{prop}
\begin{prop}\label{kmc}
Let $G$ be a finitely generated Abel group, and let $g_1, g_2\in G$ be two elements satisfying $G/\langle g_1\rangle \cong G/\langle g_2 \rangle$.
Then, there exists an isomorphism $$G\ni g_1\mapsto g_2\in  G.$$
\end{prop}
The above propositions follow from elementary arguments on finitely generated Abel groups,
and we prove them in Appendix.
We note that Prop. \ref{kmc} does not hold for non-finitely generated groups such as $G=\bigoplus_\mathbb{N}\mathbb{Z}$.
\begin{cor}\label{uB}
Let $B_1, B_2$ be unital C*-algebras with finitely generated K-groups satisfying
$$K_0(B_1)\cong K_0(B_2),\quad K_1(C_{u_{B_1}})\cong K_1(C_{u_{B_2}}).$$
Then, there exists an isomorphism $K_0(B_1)\ni [1_{B_1}]_0\mapsto [1_{B_2}]_0\in K_0(B_2)$.
In particular,
the choice of the unital UCT Kirchberg algebra $B$ in Cor. \ref{exb} is unique.
\end{cor}
\subsection{Homotopy set $[X, \operatorname{Aut}(A)]$}\label{23}
We recall M. Dadarlat's result \cite[Thm. 4.6]{D4} which asserts $$[X, \operatorname{End}(A)]=KK(C_{u_A},  SA\otimes C(X))$$ for every unital Kirchberg algebra $A$.
Every continuous map $\alpha : X\to \operatorname{End}(A)$ is identified with a $*$-homomorphism $\alpha : A\to A\otimes C(X)$ sending $a\in A$ to the function $X\ni x\mapsto \alpha_x(a)\in A$,
which is also identified with the $C(X)$-linear map $$\tilde{\alpha} :  A\otimes C(X) \ni f\otimes a\mapsto f\alpha(a)\in  A\otimes C(X).$$
Let $l$ denote the constant map $l_x={\rm id}_A$,
and we write $C\alpha : C_{u_A}\ni a(t)\mapsto \alpha_x(a(t))\in  C_{u_A}\otimes C(X)$.
Since $a(1)\in \mathbb{C}1_A$, one has $C\alpha(a(t))-Cl(a(t))\in  SA\otimes C(X)$,
and the Cuntz picture of KK-theory gives an element $\langle C\alpha, Cl\rangle\in KK(C_{u_A}, SA\otimes C(X))$ (see \cite[Sec. 3]{D4}, \cite[Sec. 17.6.]{B}).

Let $\iota_A$ denote the inclusion $SA\hookrightarrow C_{u_A}$,
and let $\beta : X\to \operatorname{End}(A)$ be another map.
For the composition $(\alpha \circ \beta )_x =\alpha_x\circ \beta_x$,
the direct computation yields
\begin{eqnarray*}
\langle C(\alpha\circ \beta), Cl\rangle &=&\langle C(\alpha\circ \beta), C\alpha\rangle +\langle C\alpha, Cl\rangle\\
&=&(\langle C\beta, Cl\rangle)\hat{\otimes}(I_S\otimes KK(\tilde{\alpha}))+\langle C\alpha, Cl\rangle\\
&=&\langle C\beta, Cl\rangle\hat{\otimes} (I_S\otimes (KK(\tilde{\alpha})-KK(\tilde{l})))+\langle C\alpha, Cl\rangle +\langle C\beta, Cl\rangle\\
&=&\langle C\beta, Cl\rangle\hat{\otimes}((KK(\iota_A)\hat{\otimes}\langle C\alpha, Cl\rangle)\otimes I_{C(X)})\hat{\otimes}(I_{SA}\otimes \Delta_X)\\
&& +\langle C\alpha, Cl\rangle+\langle C\beta, Cl\rangle,
\end{eqnarray*}
where $\Delta_X : C(X)\otimes C(X)\ni f(x, y)\mapsto f(x, x)\in C(X)$ is the diagonal map.
For $x, y\in KK(C_{u_A}, SA\otimes C(X))$,
we define a multiplication $\circ_A$ by 
$$x\circ_A y:=x+y+y\hat{\otimes}((KK(\iota_A)\hat{\otimes}x)\otimes I_{C(X)})\hat{\otimes} (I_{SA}\otimes\Delta_X).$$
Then, one has the following theorem.
\begin{thm}[{\cite[Thm. 4.6, Prop. 5.8]{D4}}]\label{DEn}
Let $X$ be a compact metrizable space,
and let $A$ be a unital Kirchberg algebra.

\end{thm}
The invertible element means the element of $[X, \operatorname{End}(A)]$ which admits both left and right inverse,
and this is an algebraic condition.
\section{The reciprocal Kirchberg algebras}
In this section,
we show several important properties of the reciprocal algebras defined in Def. \ref{REC}.
\begin{thm}\label{MT}
Let $A$ and $B$ be two unital UCT Kirchberg algebras with finitely generated K-groups.
We have the following $\colon$

\end{thm}
Statement 1 immediately follows from Lem. \ref{exb} and Cor. \ref{uB},
and the statement 2 and 3 are discussed in Sec. \ref{M2} and \ref{M3}.
\begin{rem}
Statement 2 asserts that, for a given $A$, there is only one non-trivial  $B$ sharing the same homotopy groups of automorphism groups in the category of C*-algebras with finitely generated K-groups.
However,
there are many different $B$ sharing the same homotopy groups with $A$ if we allow C*-algebras with non-finitely generated K-groups.
For example, $A=\mathcal{O}_{n+1}$ and $B=\mathbb{M}_n(\mathcal{O}_\infty)\otimes \mathbb{M}_{p^\infty}$ with $GCD(n, p)=1$ give such pairs.
\end{rem}
\begin{thm}\label{ca}
Let $A$ and $B$ be the reciprocal Kirchberg algebras.
For the mapping cone sequence
\[C_{u_A}\xrightarrow{e_A} \mathbb{C}\xrightarrow{u_A} A,\]
the Spanier--Whitehead K-duality gives another sequence

\[\xymatrix{
D(C_{u_A})& D(\mathbb{C})\ar[l]& D(A)\ar[l]\\
B\ar@{.>}[u]&\mathbb{C}\ar[l]^{u_B}\ar@{=}[u]&C_{u_B},\ar@{.>}[u]\ar[l]^{e_B}
}\]
and there exist two broken arrows which are KK-equivalences making the diagram commutative.
\end{thm}

\begin{proof}
First,
we prove the existence of the left broken arrow.
We fix the duality classes
\[\mu_A\in KK(\mathbb{C}, A\otimes D(A)),\quad \nu_A\in KK(D(A)\otimes A, \mathbb{C}),\]
\[\mu_{C_{u_A}}\in KK(\mathbb{C}, C_{u_A}\otimes D(C_{u_A})),\quad \nu_{C_{u_A}}\in KK(D(C_{u_A})\otimes C_{u_A}, \mathbb{C}).\]
Let $d_{\mu_\mathbb{C}, \nu_A}(u_A) : D(A)\to D(\mathbb{C})$ be a dual morphism defined as in Lem. \ref{b1}.
By definition, one has $K_0(B)\cong K_0(D(C_{u_A}))$,
and we have \[KK(D(\mathbb{C}), D(C_{u_A}))/\langle d_{\mu_{C_{u_A}}, \nu_\mathbb{C}}(e_A)\rangle\cong KK({C_{u_A}}, \mathbb{C})/\langle e_A\rangle \cong KK(SA, \mathbb{C})\]
by the Puppe sequence.
Since $K_0(B)/\langle [1_B]_0\rangle \cong K_1(C_{u_B})\cong KK(SA, \mathbb{C})$,
Prop. \ref{kmc} and the UCT gives the left broken arrow.

Next, we construct the right broken arrow.
By Cor. \ref{exb},
there is a KK-equivalence $\gamma\in KK(C_{u_B}, D(A))^{-1}$,
and one has
\[KK(D(A), \mathbb{C})/\langle \gamma^{-1}\hat{\otimes} e_B\rangle \cong KK(SB, \mathbb{C})\cong K_1(C_{u_A}),\]
\[KK(D(A), \mathbb{C})/\langle d_{\mu_\mathbb{C}, \nu_A}(u_A)\rangle \cong KK(\mathbb{C}, A)/\langle u_A\rangle\cong K_1(C_{u_A}).\]
By Prop. \ref{kmc},
there exists a KK-equivalence $\alpha\in KK(A, A)^{-1}$ with $d_{\mu_\mathbb{C}, \nu_A}(u_A\hat{\otimes}\alpha)=\gamma^{-1}\hat{\otimes}e_B$.
Thus,
$\gamma\hat{\otimes}d_{\mu_\mathbb{C}, \nu_A}(\alpha)\in KK(C_{u_B}, D(A))^{-1}$ gives the desired KK-equivalence.
\end{proof} 

\subsection{Proof of the statement 2.}\label{M2}
To prove the statement 2,
we need the following proposition (see Sec. \ref{nota} for notation).
\begin{prop}\label{mt}
Let $(G, \tilde{G})$ and $(H, \tilde{H})$ be finite Abel $p$-groups satisfying ($*$).
Let $f\geq 1, F\geq 0$ be integers with $f-F\in 2\mathbb{Z}$.
Then the following equation holds if and only if $(G, \tilde{G})=(H, \tilde{H})$ $\colon$
$$G^{ f}\oplus \tilde{G}^ {(F+1)}\oplus (G\otimes\tilde{G})^{ 2}\cong H^{ f}\oplus \tilde{H}^{ (F+1)}\oplus (H\otimes \tilde{H})^{ 2}.$$
\end{prop}
\begin{proof}
We prove the statement by the induction over $h=L(G)+L(\tilde{G})\geq 0$.
In the case of $h=0$,
we have $0=H^{ f}\oplus \tilde{H}^{ (F+1)}\oplus (H\otimes \tilde{H})^{ 2}$ that implies $(G, \tilde{G})=(0, 0)=(H, \tilde{H})$.
So we will prove the statement for a pair $(G, \tilde{G})$ satisfying ($*$), $L(G)+L(\tilde{G})=h+1, h\geq 0$ under the assumption that the statement holds for every pair $(G', \tilde{G}')$ satisfying ($*$), $L(G')+L(\tilde{G}')\leq h$.

One has the following two cases :
$$I)\quad m(G, \tilde{G})\in I(G)\cap I(\tilde{G}),$$
$$II)\quad m(G, \tilde{G})\not\in I(G)\cap I(\tilde{G}).$$
First,
we discuss case I).
Let $(H, \tilde{H})$ satisfy ($*$) and the equation
$$G^{ f}\oplus \tilde{G}^{ (F+1)}\oplus (G\otimes\tilde{G})^{ 2}\cong H^{ f}\oplus \tilde{H}^{ (F+1)}\oplus (H\otimes \tilde{H})^{ 2}.$$
Since $f\geq 1$ and $F+1\geq 1$,
one has $m(G, \tilde{G})=m(H, \tilde{H})=m$.
The left hand side has at least $f+(F+1)+2$ copies of $\mathbb{Z}_{p^m}$ as direct summands.
If $m\not\in I(H)\cap I(\tilde{H})$,
one has $m>k$ for every $k\in I(H)\cap I(\tilde{H})$ and it follows that $H\otimes \tilde{H}$ does not contain $\mathbb{Z}_{p^m}$ as a direct summand by the definition of ($*$).
It also follows that $H^{ f}\oplus \tilde{H}^{ (F+1)}$ has at most ${\rm max}\{f, F+1\}$ copies of $\mathbb{Z}_{p^m}$ as direct summands that implies $f+(F+1)+2\leq {\rm max}\{f, F+1\}$.
This is a contradiction and we have $m\in I(H)\cap I(\tilde{H})$.
Thus, one can write $$(G, \tilde{G})=(\mathbb{Z}_{p^m}\oplus G', \mathbb{Z}_{p^m}\oplus \tilde{G}'),\quad (H, \tilde{H})=(\mathbb{Z}_{p^m}\oplus H', \mathbb{Z}_{p^m}\oplus \tilde{H}')$$
by the pairs $(G', \tilde{G}'), (H', \tilde{H}')$ satisfying ($*$),  $L(G')+L(\tilde{G}')=h-1\leq h$, 
and the following equation holds
$$G'^{ (f+2)}\oplus \tilde{G}'^{ (F+2+1)}\oplus (G'\otimes\tilde{G}')^{ 2}\cong H'^{ (f+2)}\oplus \tilde{H}'^{(F+2+1)}\oplus (H'\otimes \tilde{H}')^{ 2}.$$
The assumption of the induction shows $(G', \tilde{G}')=(H', \tilde{H}')$ and the statement is proved for $L(G)+L(\tilde{G})=h+1$ in case I).

Next,
we discuss case II).
As in case I),
we have $m(G, \tilde{G})=m(H, \tilde{H})=m$ for a pair $(H, \tilde{H})$ satisfying ($*$) and
$$G^{ f}\oplus \tilde{G}^{ (F+1)}\oplus (G\otimes\tilde{G})^{ 2}\cong H^{ f}\oplus \tilde{H}^{ (F+1)}\oplus (H\otimes \tilde{H})^{ 2}.$$
The same argument as in case I) shows that $m\not \in I(H)\cap I(\tilde{H})$ and both $G\otimes\tilde{G}$ and $H\otimes \tilde{H}$ contain no copies of $\mathbb{Z}_{p^m}$ as direct summands.

If $m\in I(G)\backslash I(\tilde{G})$,
the left hand side of the above equation contains exactly $f$ copies of $\mathbb{Z}_{p^{m}}$ as direct summands.
If $m\not\in I(H)\backslash I(\tilde{H})$,
it means $m\in I(\tilde{H})\backslash I(H)$ and the right hand side contains exactly $F+1$ copies of $\mathbb{Z}_{p^m}$.
Since $f-F=1$ contradicts to the assumption,
one has $m\in I(H)\backslash I(\tilde{H})$.
One can write $(G, \tilde{G})=(\mathbb{Z}_{p^m}\oplus G', \tilde{G}), (H, \tilde{H})=(\mathbb{Z}_{p^m}\oplus H', \tilde{H})$ and it is easy to check that $(G', \tilde{G}), (H', \tilde{H})$ also satisfy ($*$) and
$$G'^{ f}\oplus \tilde{G}^{ (F+2+1)}\oplus (G'\otimes \tilde{G})^{ 2}\cong H'^{ f}\oplus \tilde{H}^{ (F+2+1)}\oplus (H'\otimes \tilde{H})^{ 2}.$$ 
Since $L(G')+L(\tilde{G})=h$,
the assumption of the induction yields $(G', \tilde{G})=(H', \tilde{H})$ which implies $(G, \tilde{G})=(H,\tilde{H})$.

If $m\in I(\tilde{G})\backslash I(G)$ the same argument shows $(G, \tilde{G})=(H,\tilde{H})$,
and we have proven the statement for $L(G)+L(\tilde{G})=h+1$ in case II).
This completes the induction.
\end{proof}
\begin{cor}\label{daij}
Let $(G, \tilde{G})$ and $(H, \tilde{H})$ be two pairs of finite Abel $p$-groups satisfying ($**$).
Let $f\geq 0, F\geq 0$ be integers with $f-F\in 2\mathbb{Z}$.
Then, the following equation holds if and only if $(G, \tilde{G})=(H, \tilde{H})$ $\colon$
$$G^{ f}\oplus \tilde{G}^{ (F+1)}\oplus (G\otimes\tilde{G})^{ 2}\cong H^{ f}\oplus \tilde{H}^{ (F+1)}\oplus (H\otimes \tilde{H})^{ 2}.$$

\end{cor}
\begin{proof}
Since $(G, \tilde{G})$ (resp. $(H, \tilde{H})$) satisfies ($**$),
one has $m(G, \tilde{G})\in I(\tilde{G})$ (resp. $m(H, \tilde{H})\in I(\tilde{H})$),
and $F+1\geq 1$ implies $m(G, \tilde{G})=m(H, \tilde{H})=m$.

We first consider the case $m\in I(G, p)\cap I(\tilde{G}, p)$.
The left hand side of the above equation has at least $f+(F+1)+2$ copies of $\mathbb{Z}_{p^m}$ as direct summands.
If $m\not \in I(H, p)\cap I(\tilde{H}, p)$,
the right hand side has exactly $F+1$ copies of $\mathbb{Z}_{p^m}$ as direct summands that implies $f+(F+1)+2\leq F+1$.
This is a contradiction,
and we have $m\in I(H, p)\cap I(\tilde{H}, p)$.
Thus, there exist $(G', \tilde{G}')$ and $(H', \tilde{H}')$ satisfying ($*$) and $$(G, \tilde{G})=(\mathbb{Z}_{p^m}\oplus G', \mathbb{Z}_{p^m}\oplus\tilde{G}'),\quad (H, \tilde{H})=(\mathbb{Z}_{p^m}\oplus H', \mathbb{Z}_{p^m}\oplus \tilde{H}'),$$
$$G'^{ f+2}\oplus \tilde{G}'^{ (F+2+1)}\oplus (G'\otimes\tilde{G}')^{ 2}\cong H'^{ f+2}\oplus \tilde{H}'^{ (F+2+1)}\oplus (H'\otimes \tilde{H}')^{ 2},$$
and Prop. \ref{mt} shows $(G', \tilde{G}')=(H', \tilde{H}')$.

Next, we discuss the case $m\not \in I(G, p)\cap I(\tilde{G}, p)$.
The same argument as in the previous case shows $m\not\in I(H, p)\cap I(\tilde{H}, p)$, and we write $$(G, \tilde{G})=(G, \mathbb{Z}_{p^m}\oplus\tilde{G} '),\quad (H, \tilde{H})=(H, \mathbb{Z}_{p^m}\oplus\tilde{H}'),$$
where $(G, \tilde{G}')$ and $(H, \tilde{H}')$ satisfy ($*$) and
$$G^{ f+2}\oplus \tilde{G}'^{ (F+1)}\oplus (G\otimes\tilde{G}')^{ 2}\cong H^{ f+2}\oplus \tilde{H}'^{ (F+1)}\oplus (H\otimes \tilde{H}')^{ 2}.$$
Prop. \ref{mt} shows $(G, \tilde{G}')=(H, \tilde{H}')$,
and this proves the statement.
\end{proof}

Following the notation in Section \ref{nota},
we write $$K_i(A)=\mathbb{Z}^{F_i^A}\oplus \bigoplus_p K_i(A)(p),\quad K_i(C_{u_A})=\mathbb{Z}^{f^A_i}\oplus\bigoplus_p K_i(C_{u_A})(p).$$
Note that Remark \ref{usi} implies 
$$K_i(D(A))=\mathbb{Z}^{F_i^A}\oplus \bigoplus_p K_{i+1}(A)(p),\quad K_i(D(C_{u_A}))=\mathbb{Z}^{f_i^A}\oplus\bigoplus_p K_{i+1}(C_{u_A})(p).$$
Using the mapping cone sequence $SA\to C_{u_A}\to \mathbb{C}\to A$,
one can easily check $K_0(C_{u_A})(p)\cong K_1(A)(p)$, $f_0^A\geq F_1^A, F_0^A\geq f_1^A$ and $F_1^A-f_0^A+1-F_0^A+f_1^A=0$.
The K$\rm\ddot{u}$nneth theorem yields
\begin{align*}
&K_0(A\otimes D(C_{u_A}))\\
=&\mathbb{Z}^{F_0^Af_0^A+F_1^Af_1^A}\oplus \bigoplus_{p} \left({A_p}^{F^A_1}\oplus {\tilde{A}_p}^{f^A_1}\oplus (A_p\otimes \tilde{A}_p)\oplus K_0(A)(p)^{f^A_0-F^A_1}\oplus K_1(C_{u_A})(p)^{F_0^A-f_1^A}\right),\\
&\\
&K_1(A\otimes D(C_{u_A}))\\
=&\mathbb{Z}^{F_1^Af_0^A+F_0^Af_1^A}\oplus \bigoplus_p \left(A_p^{f_1^A}\oplus \tilde{A}_p^{F_1^A}\oplus (A_p\otimes \tilde{A}_p)\oplus K_1(A)(p)^{(f_0^A-F_1^A)+(F_0^A-f_1^A)}\right)\\
=&\mathbb{Z}^{F_1^Af_0^A+F_0^Af_1^A}\oplus \bigoplus_p \left( A_p^{f_1^A}\oplus \tilde{A}_p^{F^A_1}\oplus(A_p\otimes \tilde{A}_p)\oplus K_1(A)(p)\right),
\end{align*}
where we write $$A_p:=K_1(A)(p)\oplus K_0(A)(p), \quad \tilde{A}_p:=K_1(A)(p)\oplus K_1(C_{u_A})(p).$$
If $[1_A]_0\in K_0(A)$ is a torsion element,
one has $(f_0^A-F_1^A, F_0^A-f_1^A)=(1, 0)$.
Conversely, one has $(f_0^A-F_1^A, F_0^A-f_1^A)=(0, 1)$ if $[1_A]_0\in K_0(A)$ is not a torsion.
One can easily check the following lemma.
\begin{lem}\label{ele}
Let $n, m, s, t$ be non-negative integers satisfying
$$2mn+m=2st+s,\quad n^2+n+m^2=t^2+t+s^2.$$
Then, we have $m=s, n=t$.
\end{lem}
\begin{proof}[{Proof of Theorem \ref{MT} 2.}]
Let $A$ and $B$ be unital UCT Kirchberg algebras with finitely generated K-groups satisfying
$$\pi_i(\operatorname{Aut}(A))=K_{i+1}(A\otimes D(C_{u_A}))\cong K_{i+1}(B\otimes D(C_{u_B}))=\pi_i(\operatorname{Aut}(B)), \; i\geq 1,$$
and we use the following notation
 $$A_p:=K_1(A)(p)\oplus K_0(A)(p), \quad \tilde{A}_p:=K_1(A)(p)\oplus K_1(C_{u_A})(p),$$
 $$B_p:=K_1(B)(p)\oplus K_0(B)(p), \quad \tilde{B}_p:=K_1(B)(p)\oplus K_1(C_{u_B})(p).$$
We consider the following three cases :

First,
we discuss case 1 and prove that $A$ and $B$ satisfy
$$K_i(B)=K_i(D(C_{u_A})),\quad K_i(A)=K_i(D(C_{u_B})).$$
One has $f_0^A=F_1^A+1,\; f_1^A=F_0^A$ and $F_1^B=f_0^B,\; F_0^B=f_1^B+1$,
and the comparison of ranks of the free parts yields
$$2F_0^AF_1^A+F_0^A=2f_0^Bf_1^B+f_0^B,\quad {F_1^A}^2+F_1^A+{F_0^A}^2={f_1^B}^2+f_1^B+{f_0^B}^2.$$
Lem. \ref{ele} shows $F_i^A=f_i^B,\; f_i^A=F_i^B$,
and we have
$${\tilde{A}_p}^F\oplus A_p^{F+1}\oplus (\tilde{A}_p\otimes A_p)^2\cong B_p^F\oplus \tilde{B}_p^{F+1}\oplus (B_p\otimes \tilde{B}_p)^2$$
for $F=F_1^A+f_1^A=f_1^B+F_1^B$.
Since $(\tilde{A}_p, A_p)$ and $(B_p, \tilde{B}_p)$ satisfy ($**$) by Prop. \ref{vn},
 Cor. \ref{daij} shows $$(\tilde{A}_p, A_p)=(B_p, \tilde{B}_p).$$
Thus, the assumption $K_i(A\otimes D(C_{u_A}))=K_i(B\otimes D(C_{u_B}))$ shows $$K_0(A)(p)=K_1(C_{u_B})(p),\; K_1(A)(p)=K_1(B)(p),$$ and we have
\begin{align*}
K_0(A)=&\mathbb{Z}^{F_0^A}\oplus \bigoplus_p K_0(A)(p)\\
=&\mathbb{Z}^{f_0^B}\oplus \bigoplus_p K_1(C_{u_B})(p)\\
=&K_0(D(C_{u_B})),
& \\
K_1(A)=&\mathbb{Z}^{F_1^A}\oplus \bigoplus_p K_1(A)(p)\\
=&\mathbb{Z}^{f_1^B}\oplus \bigoplus_p K_1(B)(p)\\
=&\mathbb{Z}^{f_1^B}\oplus \bigoplus_p K_0(C_{u_B})(p)\\
=&K_1(D(C_{u_B})).
\end{align*}
Similar argument shows $K_i(B)\cong K_i(D(C_{u_A})),\; i=0, 1$.

Next, we show $A=B$ in case 2.
In case 2, one has $f_0^A=F_1^A+1, \; f_1^A=F_0^A$ and $f_0^B=F_1^B+1, \; f_1^B=F_0^B$.
Comparing ranks of free parts of the homotopy groups,
one has 
$$2F_0^AF_1^A+F_0^A=2F_0^BF_1^B+F_0^B,\quad {F_1^A}^2+F_1^A+ {F_0^A}^2={F_1^B}^2+F_1^B+{F_0^B}^2,$$
and Lem. \ref{ele} implies $F_i^A=F_i^B, \; f_i^A=f_i^B$.
Thus,
we obtain
$$\tilde{A}_p^{F}\oplus A_p^{F+1}\oplus (\tilde{A}_p\otimes A_p)^2\cong \tilde{B}_p^F \oplus B_p^{F+1}\oplus (\tilde{B}_p\otimes B_p)^2$$
for $F:=F_1^A+f_1^A=F_1^B+f_1^B$.
Since $(\tilde{A}_p, A_p)$ and $(\tilde{B}_p, B_p)$ satisfy ($**$) by Prop. \ref{vn},
Cor. \ref{daij} shows $(\tilde{A}_p, A_p)=(\tilde{B}_p, B_p)$, which implies $A\sim_{KK} B,\; C_{u_A}\sim_{KK} C_{u_B}$ by a similar argument as in case 1,
and Cor. \ref{uB} gives $A\cong B$.

The same argument also shows $A=B$ in case 3,
and this completes the proof.
\end{proof}
\subsection{Proof of the statement 3.}\label{M3}
In this section,
$A$ and $B$ are unital separable UCT Kirchberg algebras satisfying
$$D(C_{u_A})\sim_{KK} B,\quad D(SA)\sim_{KK} SD(A)\sim_{KK} SC_{u_B},$$
and we denote their duality classes by
$$\mu_{CA}\in KK(\mathbb{C}, C_{u_A}\otimes B),\quad \nu_{CA}\in KK(B\otimes C_{u_A}, \mathbb{C}),$$
$$\mu_{SA}\in KK(\mathbb{C}, SA\otimes SC_{u_B}),\quad \nu_{SA}\in KK(SC_{u_B}\otimes SA, \mathbb{C}).$$
Let $\gamma \in KK(C_{u_B}, SC_{u_B}\otimes S)^{-1}$ be a KK-equivalence,
and we also denote by $\iota_B$ the natural map $B\otimes S\hookrightarrow C_{u_B}$.
Then,
as in Section \ref{23},
the semi-group $(KK(C_{u_B}, C(X)\otimes B\otimes S), {_B\circ})$ with the multiplication
$$x {_{_B}\circ} y:= x+y+y\hat{\otimes}(I_{C(X)}\otimes (KK(\iota_B)\hat{\otimes} x))\hat{\otimes}(\Delta_X\otimes I_{B\otimes S})$$
is isomorphic to $[X, \operatorname{End}(B)]$.

We may assume that $[1_A]_0$ is a torsion element and $[1_B]_0$ is not,
and we will prove $(KK(C_{u_A}, SA\otimes C(X)), \circ_A)$ and $(KK(C_{u_B}, C(X)\otimes B\otimes S), {_B\circ})$ are anti-isomorphic.
Combining Thm. \ref{DEn} and Lem. \ref{b1},
one can expect that the map
$$d^X_{\mu_{CA}, \nu_{SA}} : KK(C_{u_A}, SA\otimes C(X))\to KK(SC_{u_B}, C(X)\otimes B)$$
provides the anti-isomorphism 
which proves the statement 3.
To do this end,
we need the following lemmas.
\begin{lem}\label{36}
There exists $\alpha \in KK(B\otimes S, B\otimes S)^{-1}$ satisfying
$$(d_{\mu_{CA}, \nu_{\mathbb{C}}}(KK(e_A))\otimes I_S)\hat{\otimes}\alpha =KK(u_B)\otimes I_S\in KK(S, B\otimes S)$$
for the map $e_A : C_{u_A}\to \mathbb{C}$.
\end{lem}
\begin{proof}
The isomorphism $d_{\mu_{CA}, \nu_{\mathbb{C}}}$ and the Puppe sequence $$KK(\mathbb{C}, \mathbb{C})\xrightarrow{e_A\hat{\otimes}-} KK(C_{u_A}, \mathbb{C})\to KK(SA, \mathbb{C})\to KK(S, \mathbb{C})$$ show
$$KK(\mathbb{C}, B)/\langle d_{\mu_{CA}, \nu_{\mathbb{C}}}(KK(e_A))\rangle\cong KK(C_{u_A}, \mathbb{C})/\langle KK(e_A)\rangle\cong KK(SA, \mathbb{C}),$$
and the right hand side is isomorphic to
$$KK(\mathbb{C}, SC_{u_B})=KK(\mathbb{C}, B)/\langle KK(u_B)\rangle.$$
Thus, Prop. \ref{kmc} gives a $KK$-equivalence $\alpha\in KK(B\otimes S, B\otimes S)^{-1}$ satisfying
$$(d_{\mu_{CA}, \nu_\mathbb{C}}(KK(e_A))\otimes I_S)\hat{\otimes}\alpha =KK(u_B)\otimes I_S.$$
\end{proof}
\begin{lem}\label{37}
There exists $\beta\in KK(C_{u_B}, C_{u_B})^{-1}$ satisfying
$$KK(\iota_B)=\alpha^{-1}\hat{\otimes}\left(d_{\mu_{SA}, \nu_{CA}}(KK(\iota_A))\otimes I_S\right)\hat{\otimes}\gamma^{-1}\hat{\otimes}\beta^{-1}\in KK(B\otimes S, C_{u_B}).$$
\end{lem}
\begin{proof}
By the Puppe sequence $KK(SA, SA)\xrightarrow{-\hat{\otimes}\iota_A}KK(SA, C_{u_A})\xrightarrow{-\hat{\otimes}e_A}KK(SA, \mathbb{C})$,
the kernel of the map $-\hat{\otimes}KK(e_A)$ is $KK(SA, SA)\hat{\otimes}KK(\iota_A)\subset KK(SA, C_{u_A})$,
Lem. \ref{b1} implies that the kernel of the map
$$d_{\mu_{CA}, \nu_\mathbb{C}}(KK(e_A))\hat{\otimes}- : KK(B, SC_{u_B})\to KK(\mathbb{C}, SC_{u_B})$$
is $d_{\mu_{SA}, \nu_{CA}}(KK(\iota_A))\hat{\otimes}KK(SC_{u_B}, SC_{u_B})$.
Thus,
the kernel of the map $$(d_{\mu_{CA}, \nu_\mathbb{C}}(KK(e_A))\otimes I_S)\hat{\otimes}- : KK(B\otimes S, C_{u_B})\to KK(S, C_{u_B})$$
is equal to $(d_{\mu_{SA}, \nu_{CA}}(KK(\iota_A))\otimes I_S)\hat{\otimes}\gamma^{-1}\hat{\otimes}KK(C_{u_B}, C_{u_B})$.
By Lem. \ref{36},
one can identify the kernel of the map
$$(KK(u_B)\otimes I_S)\hat{\otimes}- : KK(B\otimes S, C_{u_B})\to KK(S, C_{u_B})$$
with $$\alpha^{-1}\hat{\otimes}(d_{\mu_{SA}, \nu_{CA}}(KK(\iota_A))\otimes I_S)\hat{\otimes}\gamma^{-1}\hat{\otimes}KK(C_{u_B}, C_{u_B}),$$
and this implies
$$KK(\iota_B)\hat{\otimes}KK(C_{u_B}, C_{u_B})=\alpha^{-1}\hat{\otimes}(d_{\mu_{SA}, \nu_{CA}}(KK(\iota_A))\otimes I_S)\hat{\otimes}\gamma^{-1}\hat{\otimes}KK(C_{u_B}, C_{u_B}).$$
There exist two elements $Z_1, Z_2\in KK(C_{u_B}, C_{u_B})$ satisfying
$$KK(\iota_B)\hat{\otimes}Z_1=\alpha^{-1}\hat{\otimes}(d_{\mu_{SA}, \nu_{CA}}(KK(\iota_A))\otimes I_S)\hat{\otimes}\gamma^{-1},$$
$$\alpha^{-1}\hat{\otimes}(d_{\mu_{SA}, \nu_{CA}}(KK(\iota_A))\otimes I_S)\hat{\otimes}\gamma^{-1}\hat{\otimes}Z_2=KK(\iota_B),$$
and we have $KK(\iota_B)\hat{\otimes}Z_1\hat{\otimes}Z_2=KK(\iota_B)$.
Since $[1_B]_0$ is not a torsion element,
the Puppe sequence shows that $-\hat{\otimes}KK(\iota_B)$ induces the following surjections
$$K_0(B\otimes S)\to K_0(C_{u_B}),\quad K_1(B\otimes S)\to K_1(C_{u_B}).$$
Since $K_i(C_{u_B})$ are finitely generated,
one can easily check that $Z_i$ induce automorphisms of $K_i(C_{u_B})$,
and the UCT proves $Z_i\in KK(C_{u_B}, C_{u_B})^{-1}$.
\end{proof}
One has a natural isomorphism of KK-groups
$$D_{A, B}^X : KK(C_{u_A}, SA\otimes C(X)) \to KK(C_{u_B}, C(X)\otimes B\otimes S)$$
defined by
$$D_{A, B}^X(x):=\beta\hat{\otimes}\gamma\hat{\otimes}(d^X_{\mu_{CA}, \nu_{SA}}(x)\otimes I_S)\hat{\otimes}(I_{C(X)}\otimes \alpha).$$
Now the following theorem proves the statement 3.
\begin{thm}\label{mul}
For any $x, y\in KK(C_{u_A}, SA\otimes C(X))$,
we have
$$D_{A, B}^X(x\circ_A y)=D_{A, B}^X(y)\; {_{_B}\circ}\; D_{A, B}^X(x).$$
\end{thm}
\begin{proof}
We write 
$$t:=y\hat{\otimes}((KK(\iota_A)\hat{\otimes}x)\otimes I_{C(X)})\in KK(C_{u_A}, SA\otimes C(X\times X))$$
and direct computation yields
\begin{align*}
&d^X_{\mu_{CA}, \nu_{SA}}(t\hat{\otimes}(I_{SA}\otimes \Delta_X))\\
=&(I_{SC_{u_B}}\otimes (\mu_{CA}\hat{\otimes}((t\hat{\otimes}(I_{SA}\otimes \Delta_X))\otimes I_B)))\hat{\otimes}(\nu_{SA}\otimes I_{C(X)\otimes B})\\
=&(I_{SC_{u_B}}\otimes(\mu_{CA}\hat{\otimes}(t\otimes I_B)))\hat{\otimes}(I_{SC_{u_B}\otimes SA}\otimes \Delta_X\otimes I_B)\hat{\otimes}(\nu_{SA}\otimes I_{C(X)\otimes B})\\
=&(I_{SC_{u_B}}\otimes (\mu_{CA}\hat{\otimes}(t\otimes I_B)))\hat{\otimes}(\nu_{SA}\otimes I_{C(X\times X)\otimes B})\hat{\otimes}(\Delta_X\otimes I_B)\\
=&d^{X\times X}_{\mu_{CA}, \nu_{SA}}(t)\hat{\otimes}(\Delta_X\otimes I_B).
\end{align*}
By Lem. \ref{b1},
one has
\begin{align*}
d^{X\times X}_{\mu_{CA}, \nu_{SA}}(t)=&d^X_{\mu_{SA}, \nu_{SA}}(KK(\iota_A)\hat{\otimes} x)\hat{\otimes}(I_{C(X)}\otimes d^X_{\mu_{CA}, \nu_{SA}}(y))\\
=&d^X_{\mu_{CA}, \nu_{SA}}(x)\hat{\otimes}(I_{C(X)}\otimes (d_{\mu_{SA}, \nu_{CA}}(KK(\iota_A))\hat{\otimes} d^X_{\mu_{CA}, \nu_{SA}}(y))).
\end{align*}
By Lem. \ref{37},
it is straightforward to check
\begin{align*}
&D_{A, B}^X(x)\hat{\otimes}(I_{C(X)}\otimes (KK(\iota_B)\hat{\otimes} D_{A, B}^X(y)))\hat{\otimes}(\Delta_X\otimes I_{B\otimes S})\\
=&D^X_{A, B}(t\hat{\otimes}(I_{SA}\otimes\Delta_X)),
\end{align*}
and this proves the statement.
\end{proof}
\begin{proof}[{Proof of the statement 3.}]
We have natural isomorphisms of the semi-groups
$$[X, \operatorname{End}(A)]\to (KK(C_{u_A}, SA\otimes C(X)), \circ_A),$$
$$(KK(C_{u_B}, C(X)\otimes B\otimes S),\; _B\circ )\to [X, \operatorname{End}(B)].$$
Thus, Thm. \ref{mul} gives a natural anti-isomorphism $[X, \operatorname{End}(A)]\to [X, \operatorname{End}(B)]$,
and this induces a natural anti-isomorphism of groups $[X, \operatorname{Aut}(A)]\to [X, \operatorname{Aut}(B)]$ by Thm. \ref{DEn}.
\end{proof}
\subsection{K-groups of the reciprocal algebras}
We describe how the K-groups of the reciprocal algebras look like with a simple example.
Let $A$ and $B$ be the reciprocal Kirchberg algebras,
where we may assume $[1_A]_0\not\in \operatorname{Tor}(K_0(A))$ and $[1_B]_0\in\operatorname{Tor}(K_0(B))$.
Combining the previous results,
K-groups of $A$ has the following presentation :
\[K_0(A)=T_0\oplus\mathbb{Z}\oplus\mathbb{Z}^{F_0},\quad [1_A]_0=(t, n, 0),\quad K_1(A)=T_1\oplus\mathbb{Z}^{F_1},\]
where $T_0, T_1$ are finite Abel groups and $F_0, F_1\geq 0$, $n>0$ are integers.
We write $\tilde{T}_0:=(T_0\oplus\mathbb{Z})/\langle (t, n)\rangle$,
and K-groups of $B$ are given by
\[K_0(B)=\tilde{T}_0\oplus \mathbb{Z}^{F_1},\quad [1_B]_0=(\tilde{t}, 0),\quad K_1(B)=T_1\oplus\mathbb{Z}^{F_0}.\]
The group $\tilde{T}_0$ is a finite Abel group and $\tilde{t}$ is an element satisfying $\tilde{T}_0/\langle \tilde{t}\rangle\cong T_0$.
Note that such $\tilde{t}$ is unique up to automorphisms of $\tilde{T}_0$ by Prop. \ref{kmc}.
Now we have the following group theoretic consequence.
\begin{cor}
Let $T$ and $\tilde{T}$ be finite Abel groups,
and let $n>0$ be an integer.

\end{cor}
\begin{rem}
As a direct consequence of Appendix,
one can verify that the above $(t, n)$ and $\tilde{t}$ in Statement 1. exist for a given pair $(T, \tilde{T})$ if and only if $(T(p), \tilde{T}(p))$ satisfies the condition $(**)$ for every prime $p$ (see Sec. \ref{nota}).

The elements $u, v$  in Statement 2. correspond to $(t, n)\in T\oplus\mathbb{Z}$ and a lift of $\tilde{t}\in (T\oplus\mathbb{Z})/\langle (t, n)\rangle$.
There might be some algebraic relation characterizing $u$ and $v$,
but we do not know how to characterize them at present.
\end{rem}
\begin{ex}
To give more concrete presentation of $t\in T$ and $\tilde{t}\in \tilde{T}$,
we consider the case $T=T(p), n=p^l$ for simplicity.
After choosing suitable presentation we may assume
\[t=(0, p^{k_1-r_1},\cdots, p^{k_s-r_s})\in (\mathbb{Z}_{p^{n_1}}\oplus\cdots\oplus\mathbb{Z}_{p^{n_h}})\oplus\mathbb{Z}_{p^{k_1}}\oplus\cdots\oplus\mathbb{Z}_{p^{k_s}}\]
where the integers $k_i, r_i$ satisfy
\[0\leq k_1-r_1<k_1<k_2-r_2+r_1<k_2<\cdots<k_{s-1}<k_s-r_s+r_{s-1}<k_s<l+r_s\]
(see Lem. \ref{G1}).
Then,
one has
\[\tilde{T}=(\mathbb{Z}_{p^{n_1}}\oplus\cdots\oplus\mathbb{Z}_{p^{n_h}})\oplus\mathbb{Z}_{p^{k_1-r_1}}\oplus\mathbb{Z}_{p^{k_2-r_2+r_1}}\oplus\cdots\oplus\mathbb{Z}_{p^{k_s-r_s+r_{s-1}}}\oplus\mathbb{Z}_{p^l+r_s},\]
and $\tilde{t}$ is given by
\[(0, 1, p^{r_1}, \cdots, p^{r_{s-1}}, p^{r_s}).\]
\end{ex}
\section{Spanier--Whitehead duality for bundles of C*-algebras}\label{spwh}
In this section,
we show that the Spanier--Whitehead $K_X$-duality holds for locally trivial continuous $C(X)$-algebras with  fiber $A$,
where $X$ is a finite CW-complex and $A$ is a separable nuclear UCT C*-algebra with finitely generated K-groups.

Let $\mathcal{A}$ be a locally trivial continuous $C(X)$-algebra with fiber $A$,
and we may assume that $A$ is a stable Kirchberg algebra by Thm. \ref{de}.
Let $D$ be the stable Kirchberg algebra which is K-dual to $A$ with the $*$-homomorphisms
\[\mu_0 : \mathcal{O}_\infty\otimes\mathbb{K}\to D\otimes A,\quad \nu_0 : A\otimes D\to \mathcal{O}_\infty\otimes\mathbb{K}\]
providing the duality classes.
In this section,
we frequently use Thm. \ref{KGE} without mentioning it.

Let $X$ be a finite CW-complex obtained by attaching $\mathbb{D}^{d+1}$ to another finite CW-complex $Y$ via $\theta : S^d\to Y$.
We assume that there exists a locally trivial continuous $C(Y)$-algebra $\mathcal{D}_Y$ with fiber $D$ which is $K_Y$-dual to $\mathcal{A}(Y)$.
For example,
if $Y$ is the $0$-skeleton,
Thm. \ref{SWK} shows that the above assumption holds.
Then, we will construct the $K_X$-dual of $\mathcal{A}$ and their duality classes.

We consider the case $X=Y\cup \mathbb{D}^{d+1}, \;Y\cap \mathbb{D}^{d+1}=S^d$ (i.e., the attaching map $\theta : S^d\to Y$ is injective).
By the lemma below,
the general case is reduced to the above easy case.
\begin{lem}\label{L2}
Let $X$ and $Y$ be finite CW-complexes,
and let $\theta : S^d\to Y$ be a general map attaching $\mathbb{D}^{d+1}$ to $Y$.
Assume that $\mathcal{A}(Y)$ satisfies $K_Y$-duality.
We write
\[Y_{1/2}:=([1/2, 1]\times S^d) \sqcup Y/(1, x)\sim \theta (x),\quad x\in S^d,\]
\[\mathbb{D}^{d+1}:=[0, 1/2]\times S^d/(0, x)\sim (0, z),\quad x, z\in S^d.\]
Then, we have $X=Y_{1/2}\cup \mathbb{D}^{d+1}$ with $Y_{1/2}\cap \mathbb{D}^{d+1}=S^d$,
and $\mathcal{A}(Y_{1/2})$ also satisfies $K_{Y_{1/2}}$-duality.
\end{lem}
\begin{proof}
Since the map $r : Y_{1/2}\to Y$ defined by
\[r(y)=y, \; y\in Y,\quad r(t, x)=\theta (x),\; (t, x)\in Y_{1/2}\]
is a deformation retract,
the pull-back \[r^* : [Y, \operatorname{BAut}(A)]\ni [\mathcal{A}]\mapsto [C(Y_{1/2})\otimes_{C(Y)}\mathcal{A}(Y)]\in [Y_{1/2}, \operatorname{BAut}(A)]\]
is a bijection, where $C(Y_{1/2})$ is regarded as a $C(Y)$-algebra by $r^* : C(Y)\to C(Y_{1/2})$ and the inverse map of $r^*$ is given by the restriction $\pi_Y$ with $\pi_Y(\mathcal{A}(Y_{1/2}))=\mathcal{A}(Y)$.
Thus, we have $\mathcal{A}(Y_{1/2})\cong C(Y_{1/2})\otimes_{C(Y)}\otimes \mathcal{A}(Y)$,
and it is easy to check that $C(Y_{1/2})\otimes_{C(Y)}\mathcal{D}_Y$ is $K_{Y_{1/2}}$-dual to $\mathcal{A}(Y_{1/2})$.
\end{proof}

The restriction $\mathcal{A}(\mathbb{D}^{d+1})$ is a bundle over the contractible space $\mathbb{D}^{d+1}$,
and one has $\mathcal{A}(\mathbb{D}^{d+1})\cong C(\mathbb{D}^{d+1})\otimes A$ which restricts to the $C(S^d)$-linear isomorphism
\[\varphi_A : \mathcal{A}(Y\cap \mathbb{D}^{d+1})\to C(S^d)\otimes A.\]
Since $\mathcal{A}(Y\cap \mathbb{D}^{d+1})\cong C(S^d)\otimes A$,
we have an isomorphism \[\psi'_D : \mathcal{D}_Y(Y\cap \mathbb{D}^{d+1})\to C(S^d)\otimes D\] by Thm. \ref{Ddeq} and uniqueness of the dual algebra.
This $\psi'_D$ is a candidate of  the clutching function to extend $\mathcal{D}_Y$ onto $X$.
We have the following presentations : \[C(X)=\{ (f, g)\in C(Y)\oplus C(\mathbb{D}^{d+1})\; |\; \pi_{Y\cap \mathbb{D}^{d+1}}(f)=\pi_{S^d}(g)\},\]
\[ \mathcal{A}=\{(a, b)\in \mathcal{A}(Y)\oplus C(\mathbb{D}^{d+1})\otimes A\; |\; \pi_{{Y}\cap \mathbb{D}^{d+1}}(a)=\varphi_A^{-1}(\pi_{S^d}(b))\},\]
and the $C(Y)$-linear $*$-homomorphisms representing the $K_Y$-duality classes :
\[\mu : C(Y_{})\otimes\mathcal{O}_\infty\otimes\mathbb{K}\to \mathcal{D}_Y\otimes_{C(Y_{})}\mathcal{A}(Y_{}),\]
\[\nu : \mathcal{A}(Y_{})\otimes_{C(Y_{})}\mathcal{D}_Y\to C(Y_{})\otimes\mathcal{O}_\infty\otimes\mathbb{K}.\]
We will modify $\psi'_D$ so that there are no obstructions to extending $\mu$ and $\nu$ onto $X$.
By Lem. \ref{b} 4.,
the both of 
\[{\rm id}_{C(S^d)}\otimes \nu_0,\quad \nu\upharpoonright_{S^d}\circ (\varphi_A^{-1}\otimes{\psi'_D}^{-1})\] give $K_{S^d}$-duality classes for $C(S^d)\otimes A$ and $C(S^d)\otimes D$,
and Lem. \ref{b} 2. shows that there exists a $C(S^d)$-linear isomorphism $\beta$ of $C(S^d)\otimes D$ satisfying \[KK_{S^d}(\nu\upharpoonright_{S^d}\circ(\varphi_A^{-1}\otimes ({\psi'_D}^{-1}\circ \beta)))=KK_{S^d}({\rm id}_{C(S^d)}\otimes \nu_0).\]
For the isomorphism $\psi_D :=\beta^{-1}\circ \psi'_D : \mathcal{D}_Y(Y\cap\mathbb{D}^{d+1})\to C(S^d)\otimes D$,
one has a continuous path of $C(S^d)$-linear homomorphisms
\[\nu_t : C(S^d)\otimes A\otimes D\to C(S^d)\otimes\mathcal{O}_\infty\otimes\mathbb{K},\quad \nu_0={\rm id}_{C(S^d)}\otimes \nu_0,\quad \nu_{1}=\nu\upharpoonright_{S^d}\circ(\varphi_A\otimes\psi_D)^{-1}\]
for $t\in [0, 1]$,
and we define the following nuclear separable continuous $C(X)$-algebra
\[\mathcal{D}_X :=\{ (a, b)\in \mathcal{D}_Y\oplus C(\mathbb{D}^{d+1})\otimes D\; |\; \pi_{Y_{}\cap\mathbb{D}^{d+1}}(a)=\psi_D^{-1}(\pi_{S^d}(b))\}.\]
Now one has
\begin{align*}
&\mathcal{A}\otimes_{C(X)}\mathcal{D}_X\\
=&\{(a, b)\in (\mathcal{A}(Y_{})\otimes_{C(Y_{})}\mathcal{D}_Y)\oplus (C(\mathbb{D}^{d+1})\otimes A\otimes D)\; |\; \pi_{Y_{}\cap\mathbb{D}^{d+1}}(a)=(\varphi_A^{-1}\otimes\psi_D^{-1})(\pi_{S^d}(b))\},\\
&\mathcal{D}_X\otimes_{C(X)}\mathcal{A}\\
=&\{(a, b)\in (\mathcal{D}_Y\otimes_{C(Y_{})}\mathcal{A}(Y_{}))\oplus (C(\mathbb{D}^{d+1})\otimes D\otimes A)\; |\; \pi_{Y_{}\cap\mathbb{D}^{d+1}}(a)=(\psi_D^{-1}\otimes\varphi_A^{-1})(\pi_{S^d}(b))\}.
\end{align*}
\begin{thm}\label{ind}
The algebras $\mathcal{D}_X$ and $\mathcal{A}$ are $K_X$-dual.
\end{thm}
By identifying $\pi_{S^d} : C(\mathbb{D}^{d+1})\to C(S^d)$ with \[\{g(t, -)\in C([0, 1]\times S^d)\;|\; g(0, -)\in \mathbb{C}1_{C(S^d)}\}\ni g(t, -)\mapsto g(1,-)\in C(S^d),\]
we can extend $\nu$ to the $C(X)$-linear homomorphism
\[\tilde{\nu} : \mathcal{A}\otimes_{C(X)}\mathcal{D}_X\ni (a, b(t, -))\mapsto (\nu(a), \nu_t(b(t, -)))\in C(X)\otimes\mathcal{O}_\infty\otimes\mathbb{K}.\]
\begin{lem}\label{L5}
The composition $(\psi_D\otimes\varphi_A)\circ\mu\upharpoonright_{S^d}$ is homotopic to
${\rm id}_{C(S^d)}\otimes \mu_0$.
\end{lem}
\begin{proof}
Direct computation yields
\begin{align*}
&KK_{S^d}((\psi_D\otimes\varphi_A)\circ\mu\upharpoonright_{S^d})\\
=&((\psi_D\otimes\varphi_A)\circ\mu\upharpoonright_{S^d})\hat{\otimes}(I_{C(S^d)\otimes D(A)}\otimes((I_{C(S^d)\otimes A}\otimes ({\rm id}_{C(S^d)}\otimes\mu_0))\hat{\otimes}\\
&(({\rm id}_{C(S^d)}\otimes\nu_0)\otimes I_{C(S^d)\otimes A})))\\
=&({\rm id}_{C(S^d)}\otimes\mu_0)\hat{\otimes}(\psi_D^{-1}\otimes I_{C(S^d)\otimes A})\hat{\otimes}(\mu\upharpoonright_{S^d}\otimes I_{(\mathcal{D}(S^d)\otimes_{C(S^d)}C(S^d)\otimes A)})\hat{\otimes}\\
&(\psi_D\otimes((\varphi_A\otimes\psi_D)\hat{\otimes}({\rm id}_{C(S^d)}\otimes \nu_0))\otimes I_{C(S^d)\otimes A})\\
=&({\rm id}_{C(S^d)}\otimes\mu_0)\hat{\otimes}(\psi_D^{-1}\otimes I_{C(S^d)\otimes A})\hat{\otimes}(\mu\upharpoonright_{S^d}\otimes I_{\mathcal{D}(S^d)}\otimes I_{C(S^d)\otimes A})\hat{\otimes}\\
&(I_{\mathcal{D}(S^d)}\otimes(({\rm id}_{C(S^d)}\otimes\nu_0)\circ(\varphi_A\otimes\psi_D))\otimes I_{C(S^d)\otimes A})\hat{\otimes}(\psi_D\otimes I_{C(S^d)\otimes A}),
\end{align*}
and the equation $KK_{S^d}(\nu\upharpoonright_{S^d})=KK_{S^d}(({\rm id}_{C(S^d)}\otimes\nu_0)\circ(\varphi_A\otimes\psi_D))$ implies
\[KK_{S^d}((\psi_D\otimes\varphi_A)\circ \mu\upharpoonright_{S^d})=KK_{S^d}({\rm id}_{C(S^d)}\otimes\mu_0).\]
\end{proof}

By the above lemma,
we can also find a map
\[\tilde{\mu} : C(X)\otimes\mathcal{O}_\infty\otimes\mathbb{K}\to\mathcal{D}_X\otimes_{C(X)}\mathcal{A}\]
and it is easy to check that $\tilde{\mu}_x$ (resp. $\tilde{\nu}_x$) is homotopic to either $\mu_x, x\in Y$ or $\mu_0$ (resp. $\nu_x$ or $\nu_0$) for every $x\in X$.
Thus, one has
\[(I_{\mathcal{A}_x}\otimes \tilde{\mu}_x)\hat{\otimes}(\tilde{\nu}_x\otimes I_{\mathcal{A}_x})\in KK(\mathcal{A}_x, \mathcal{A}_x)^{-1},\]
\[(\tilde{\mu}_x\otimes I_{(\mathcal{D}_X)_x})\hat{\otimes}(I_{(\mathcal{D}_X)_x}\otimes\tilde{\nu}_x)\in KK((\mathcal{D}_X)_x, (\mathcal{D}_X)_x)^{-1}\]
by Lem. \ref{b}, 1., and Thm. \ref{Ddeq} shows 
\[(I_\mathcal{A}\otimes\tilde{\mu})\hat{\otimes}(\tilde{\nu}\otimes I_\mathcal{A})\in KK_X(\mathcal{A}, \mathcal{A})^{-1},\]
\[(\tilde{\mu}\otimes I_{\mathcal{D}_X})\hat{\otimes}(I_{\mathcal{D}_X}\otimes\tilde{\nu})\in KK_X(\mathcal{D}_X, \mathcal{D}_X)^{-1}.\]

Finally, Lem. \ref{b}, 1. proves Thm. \ref{ind},
and the following is obtained by induction.
\begin{cor}\label{CXd}
Let $C$ be a separable nuclear UCT C*-algebra with finitely generated K-groups,
and let $X$ be a finite CW-complex.
For every locally trivial continuous $C(X)$-algebra $\mathcal{C}$ with fiber $C$,
there exists a locally trivial continuous $C(X)$-algebra $\mathcal{D}(\mathcal{C})$ with fiber $D(C)$ which is  $K_X$-dual to $\mathcal{C}$.
\end{cor}
\begin{rem}
Since $K_X$-duals are uniquely determined up to $KK_X$-equivalence,
Thm. \ref{Ddeq} implies that if both of $C$ and $D(C)$ are stable Kirchberg algebras
the map \[[X, \operatorname{BAut}(C)]\ni [\mathcal{C}]\mapsto [\mathcal{D}(\mathcal{C})]\in [X, \operatorname{BAut}(D(C))]\]
is a well-defined bijection.
In particular,
for $C=D(C)=\mathcal{O}_\infty\otimes\mathbb{K}$,
the bijection is exactly equal to
\[E^1_{\mathcal{O}_\infty}(X)\ni c\mapsto -c\in E^1_{\mathcal{O}_\infty}(X).\] 
\end{rem}
\begin{rem}
It might be meaningful to relate the above bijection to the Izumi--Matui's invariant for the group action $G$ (i.e., to investigate $[\operatorname{B}G, \operatorname{BAut}(C)]\to [\operatorname{B}G, \operatorname{BAut}(D(C))]$ in terms of the group actions).
However,
things seem to be not so easy even in the simplest case $C=\mathcal{O}_\infty^s$.
Moreover, the condition of having finitely generated K-groups has never appeared in the theory of group actions.
\end{rem}

In the rest of this section,
we prove the above lemmas.


\section{Bundles of the reciprocal algebras}\label{5}
In this section,
we prove our main result that the reciprocal Kirchberg algebras $A$ and $B$ share the same structure of their bundles (see Thm. \ref{taio}).
\subsection{Construction of the map $R_{A, B}$}
Let $\mathcal{A}$ be an arbitrary locally trivial continuous $C(X)$-algebra with fiber $A$.
In the Puppe sequence
$$C_{u_\mathcal{A}}\xrightarrow{e_\mathcal{A}} C(X)\xrightarrow{u_\mathcal{A}} \mathcal{A},$$
the mapping cone $C_{u_\mathcal{A}}$ is a locally trivial continuous $C(X)$-algebra with fiber $C_{u_A}$.
By Thm. \ref{de} and Cor. \ref{CXd}, there exists a locally trivial continuous $C(X)$-algebra $\mathcal{B}^s$ with fiber $B\otimes\mathbb{K}$ such that $C_{u_\mathcal{A}}$ and $\mathcal{B}^s$ are $K_X$-dual with duality classes
$$\mu\in KK_X(C(X), C_{u_\mathcal{A}}\otimes_{C(X)}\mathcal{B}^s),\quad \nu\in KK_X(\mathcal{B}^s\otimes_{C(X)}C_{u_\mathcal{A}}, C(X)).$$
Since $KK_X(C(X), \mathcal{B}^s)=K_0(\mathcal{B}^s)$ (see \cite[Proof of Cor. 2.8.]{D3}),
one obtains a properly infinite full projection $p\in \mathcal{B}^s$ with $[p]_0=\mu\hat{\otimes}(e_\mathcal{A}\otimes I_{\mathcal{B}^s})\in K_0(\mathcal{B}^s)=KK_X(C(X), \mathcal{B}^s)$.
\begin{lem}\label{bij}
For every $x\in X$,
we have $\pi_x(p)\mathcal{B}^s_x\pi_x(p)\cong B$.
Thus, $\mathcal{B}:=p\mathcal{B}^s p$ is a locally trivial continuous $C(X)$-algebra with fiber $B$.
\end{lem}
\begin{proof}
Since $\mathcal{B}^s$ is locally trivial,
the local triviality of $\mathcal{B}$ immediately follows.
So we prove $\pi_x(p)\mathcal{B}^s_x\pi_x(p)\cong B$.
By the evaluation map,
one has $$[\pi_x(p)]_0=\mu_x\hat{\otimes}(e_{u_{\mathcal{A}_x}}\otimes I_{\mathcal{B}^s_x})\in KK(\mathbb{C}, \mathcal{B}^s_x).$$
Cor. \ref{pd} shows 
$$KK(\mathbb{C}, \mathcal{B}^s_x)/\langle [\pi_x(p)]_0 \rangle \cong KK(C_{u_{\mathcal{A}_x}}, \mathbb{C})/\langle KK(e_{\mathcal{A}_x}) \rangle,$$
and the right hand side is isomorphic to $$KK(S\mathcal{A}_x, \mathbb{C})\cong KK(SA, \mathbb{C})\cong K_1(C_{u_B})$$
by the Puppe sequence
$$KK(\mathbb{C}, \mathbb{C})\xrightarrow{e_{\mathcal{A}_x}\hat{\otimes}-}KK(C_{u_{\mathcal{A}_x}}, \mathbb{C})\xrightarrow{\iota_{\mathcal{A}_x}\hat{\otimes}-}KK(S\mathcal{A}_x, \mathbb{C})\to KK(S, \mathbb{C}).$$
Since $\mathcal{B}^s_x \cong B\otimes\mathbb{K}$,
we have $$K_0(B\otimes\mathbb{K})/\langle [\pi_x(p)]_0\rangle \cong K_1(C_{u_B})\cong K_0(B\otimes\mathbb{K})/\langle [1_B]_0 \rangle,$$
and Prop. \ref{kmc} proves the statement.
\end{proof}
\begin{prop}\label{wel}
Let $\mathcal{A}_i,\; i=1, 2$ be locally trivial continuous $C(X)$-algebras with $C(X)$-linear isomorphism $\varphi : \mathcal{A}_1\to \mathcal{A}_2$,
and let $\mathcal{B}_i, \; i=1, 2$ be the $C(X)$-algebras obtained from $\mathcal{A}_i$ as in Lem. \ref{bij}.
Then, $\mathcal{B}_1$ and $\mathcal{B}_2$ are $C(X)$-linearly isomorphic.
\end{prop}
\begin{proof}
Applying Lem. \ref{RMN} to the following diagram
$$\xymatrix{S\mathcal{A}_1\ar[d]^{I_S\otimes \varphi}\ar[r]&C_{u_{\mathcal{A}_1}}\ar[r]^{e_{\mathcal{A}_1}}&C(X)\ar[r]^{u_{\mathcal{A}_1}}\ar@{=}[d]&\mathcal{A}_1\ar[d]^{\varphi}\\
S\mathcal{A}_2\ar[r]&C_{u_{\mathcal{A}_2}}\ar[r]^{e_{\mathcal{A}_2}}&C(X)\ar[r]^{u_{\mathcal{A}_2}}&\mathcal{A}_2,
}$$
one has a $KK_X$-equivalence $\gamma : C_{u_{\mathcal{A}_1}}\to C_{u_{\mathcal{A}_2}}$ making the diagram commutes.
We take the $C(X)$-algebras $\mathcal{B}^s_i$ and duality classes $\mu_i \in KK_X(C(X), C_{u_{\mathcal{A}_i}}\otimes_{C(X)} \mathcal{B}^s_i)$ for $i=1, 2$.
By Lem. \ref{b},
one has a $KK_X$-equivalence $\mathcal{B}^s_1\xrightarrow{\alpha} \mathcal{B}^s_2$.
Let $p_i$ denote the projection defined by $[p_i]_0=\mu_i\hat{\otimes}(e_{\mathcal{A}_i}\otimes I_{\mathcal{B}^s_i})$.
Direct computation yields
\begin{align*}
[p_1]_0\hat{\otimes}\alpha=&\mu_1\hat{\otimes}(e_{\mathcal{A}_1}\otimes I_{\mathcal{B}^s_1})\hat{\otimes}\alpha\\
=&\mu_1\hat{\otimes}(\gamma \otimes I_{\mathcal{B}^s_1})\hat{\otimes}(e_{\mathcal{A}_2}\otimes I_{\mathcal{B}^s_1})\hat{\otimes}(I_{C(X)}\otimes \alpha)\\
=&\mu_1\hat{\otimes}(\gamma \otimes I_{\mathcal{B}^s_1})\hat{\otimes}(I_{C_{u_{\mathcal{A}_2}}}\otimes \alpha )\hat{\otimes}(e_{\mathcal{A}_2}\otimes I_{\mathcal{B}^s_2}).
\end{align*}
Lem. \ref{b} shows $\mu_1\hat{\otimes}(\gamma \otimes I_{\mathcal{B}^s_1})\hat{\otimes}(I_{C_{u_{\mathcal{A}_2}}}\otimes \alpha )$ also provides duality classes and one has
$$\mu_1\hat{\otimes}(\gamma \otimes I_{\mathcal{B}^s_1})\hat{\otimes}(I_{C_{u_{\mathcal{A}_2}}}\otimes \alpha )=\mu_2\hat{\otimes}(I_{C_{u_{\mathcal{A}_2}}}\otimes \alpha')$$
for some $\alpha'\in KK_X(\mathcal{B}^s_2, \mathcal{B}^s_2)^{-1}$.
Finally,
we have $[p_1]_0\hat{\otimes}\alpha=[p_2]_0\hat{\otimes}\alpha' \in K_0(\mathcal{B}^s_2)$,
and this shows
$$\mathcal{B}_1=p_1\mathcal{B}^s_1p_1\cong p_2\mathcal{B}^s_2p_2=\mathcal{B}_2.$$
\end{proof}
By Prop. \ref{wel},
we obtain a well-defined natural map 
$${R_{A, B}} : [X, \operatorname{BAut}(A)]\ni [\mathcal{A}]\mapsto [\mathcal{B}]\in [X, \operatorname{BAut}(B)]$$
for the reciprocal Kirchberg algebras $A$ and $B$,
and we show the bijectivity of $R_{A, B}$ in the same spirit as \cite[Sec. 2]{D2}.
\subsection{Proof of Thm. \ref{RB}}
\begin{thm}\label{taio}
The map $R_{A, B}$ is bijective.
\end{thm}
\begin{proof}
First, we prove the injectivity.
Take two Puppe sequences $C_{u_{\mathcal{A}_i}}\xrightarrow{e_{\mathcal{A}_i}} C(X)\xrightarrow{u_{\mathcal{A}_i}} \mathcal{A}_i, i=1, 2$ and $C(X)$-algebras $\mathcal{B}^s_i$ with duality classes $$\mu_i \in KK_X(C(X), C_{u_{\mathcal{A}_i}}\otimes_{C(X)}\mathcal{B}^s_i),\quad \nu_i\in KK_X(\mathcal{B}^s_i\otimes_{C(X)}C_{u_{\mathcal{A}_i}}, C(X)).$$
The unital $C(X)$-algebra $\mathcal{B}_i$ with $[\mathcal{B}_i]={R_{A, B}}([\mathcal{A}_i])$ is obtained as a corner of $\mathcal{B}^s_i$ and the following diagram commutes
$$\xymatrix{C(X)\ar[r]^{\mu_i\hat{\otimes}(e_{\mathcal{A}_i}\otimes I_{\mathcal{B}^s_i})}&\mathcal{B}^s_i\\
C(X)\ar@{=}[u]\ar[r]^{u_{\mathcal{B}_i}}&\mathcal{B}_i,\ar[u]
}$$
where the right vertical map is the inclusion providing a $KK_X$-equivalence.
We will show $\mathcal{A}_1\cong \mathcal{A}_2$ under the assumption that there is a $C(X)$-linear isomorphism $\mathcal{B}_1\to\mathcal{B}_2$.
By assumption,
one has a $KK_X$-equivalence $\alpha$ making the following diagram commute
$$\xymatrix{C(X)\quad\ar[r]^{\mu_1\hat{\otimes}(e_{\mathcal{A}_1}\otimes I_{\mathcal{B}^s_1})}\ar@{=}[d]&\quad\mathcal{B}^s_1\ar[d]^{\alpha}\\
C(X)\quad\ar[r]^{\mu_2\hat{\otimes}(e_{\mathcal{A}_2}\otimes I_{\mathcal{B}^s_2})}&\quad\mathcal{B}^s_2,}$$
and Lem. \ref{b} also gives $\beta\in KK_X(C_{u_{\mathcal{A}_1}}, C_{u_{\mathcal{A}_2}})^{-1}$.

Let $f_i$ denote the element $(\mu_i\hat{\otimes}(e_{\mathcal{A}_i}\otimes I_{\mathcal{B}^s_i}))\otimes I_{C_{u_{\mathcal{A}_i}}}$,
and direct computation yields
\begin{align*}
f_1\hat{\otimes}(\alpha\otimes I_{C_{u_{\mathcal{A}_1}}})\hat{\otimes}(I_{\mathcal{B}^s_2}\otimes\beta)=&((\mu_2\hat{\otimes}(e_{\mathcal{A}_2}\otimes I_{\mathcal{B}^s_2}))\otimes I_{C_{u_{\mathcal{A}_1}}})\hat{\otimes}(I_{\mathcal{B}^s_2}\otimes\beta)\\
=&\beta\hat{\otimes}f_2.
\end{align*}
The element $\nu_1':=(\alpha\otimes I_{C_{u_{\mathcal{A}_1}}})\hat{\otimes}(I_{\mathcal{B}^s_2}\otimes \beta)\hat{\otimes}\nu_2$ provides duality classes for $C_{u_{\mathcal{A}_1}}$ and $\mathcal{B}^s_1$ by Lem. \ref{b},
and one can find a $\gamma \in KK_X(C_{u_{\mathcal{A}_1}}, C_{u_{\mathcal{A}_1}})^{-1}$ with
$$\nu_1=(I_{\mathcal{B}^s_1}\otimes\gamma)\hat{\otimes}\nu'_1.$$
Now one has the following commutative diagram
$$\xymatrix{C_{u_{\mathcal{A}_1}}\ar[r]^{f_1}\ar[d]^{\gamma}&\mathcal{B}_1^s\otimes_{C(X)}C_{u_{\mathcal{A}_1}}\ar[r]^{\nu_1}\ar[d]^{I\otimes\gamma}&C(X)\ar@{=}[d]\\
C_{u_{\mathcal{A}_1}}\ar[d]^{\beta}\ar[r]^{f_1}&\mathcal{B}_1^s\otimes_{C(X)}C_{u_{\mathcal{A}_1}}\ar[d]^{(\alpha\otimes I)\hat{\otimes}(I\otimes \beta)}\ar[r]^{\nu'_1}&C(X)\ar@{=}[d]\\
C_{u_{\mathcal{A}_2}}\ar[r]^{f_2}&\mathcal{B}^s_2\otimes_{C(X)}C_{u_{\mathcal{A}_2}}\ar[r]^{\nu_2}&C(X).
}$$
Since $f_i\hat{\otimes}\nu_i=e_{\mathcal{A}_i}$,
we can apply Lem. \ref{RMN} for 
$$\xymatrix{SC(X)\ar@{=}[d]\ar[r]^{I_S\otimes u_{\mathcal{A}_1}}&S\mathcal{A}_1\ar[r]&C_{u_{\mathcal{A}_1}}\ar[d]^{\gamma\hat{\otimes}\beta}\ar[r]^{e_{\mathcal{A}_1}}&C(X)\ar@{=}[d]\\
SC(X)\ar[r]^{I_S\otimes u_{\mathcal{A}_2}}&S\mathcal{A}_2\ar[r]&C_{u_{\mathcal{A}_2}}\ar[r]^{e_{\mathcal{A}_2}}&C(X),}$$
and obtain the isomorphism $\mathcal{A}_1\cong \mathcal{A}_2$.

Next,
we show the surjectivity.
Fix an arbitrary locally trivial continuous $C(X)$-algebra $\mathcal{B}$ with fiber $B$,
and we construct $[\mathcal{A}]\in [X, \operatorname{BAut}(A)]$ with $[\mathcal{B}]=R_{A, B}([\mathcal{A}])$.
There exists a separable nuclear locally trivial continuous $C(X)$-algebra $\mathcal{D}$ such that $\mathcal{D}$ and $\mathcal{B}$ are $K_X$-dual with duality classes
\[\mu\in KK_X(C(X), \mathcal{D}\otimes_{C(X)}\mathcal{B}), \quad\nu \in KK_X(\mathcal{B}\otimes_{C(X)}\mathcal{D}, C(X)).\]
Thanks to Thm. \ref{de}, 
one has a unital nuclear locally trivial continuous $C(X)$-algebra $\mathcal{D}^\sharp$ ($\sim_{KK_X}\mathcal{D}$) whose fiber $(\mathcal{D}^\sharp)_x (\sim_{KK} D(B)=C_{u_A})$ is a Kirchberg algebra. 
By Thm. \ref{KGE},
there exists a $C(X)$-linear $*$-homomorphism $d : \mathcal{D}^\sharp\to C(X)\otimes\mathcal{O}_\infty$ making the following diagram commute
\[\xymatrix{
\mathcal{D}\ar[r]^{(u_{\mathcal{B}}\otimes I_\mathcal{D})\hat{\otimes}\nu}\ar[d]&C(X)\ar[d]\\
\mathcal{D}^\sharp\ar[r]^{d}&C(X)\otimes\mathcal{O}_\infty,
}\]
where every vertical arrow is a $KK_X$-equivalence.
The map $d$ gives the Puppe sequence $SC(X)\otimes\mathcal{O}_\infty\xrightarrow{\iota} C_d\to \mathcal{D}^\sharp\to C(X)\otimes\mathcal{O}_\infty$ with a separable nuclear locally trivial continuous $C(X)$-algebra $C_d$.
We show ${C_d}_x\sim_{KK}SA$.
Since $KK(d_x)\in KK((\mathcal{D}^\sharp)_x, \mathcal{O}_\infty)$ is identified with $(u_{\mathcal{B}_x}\otimes I_{\mathcal{D}_x})\hat{\otimes}\nu_x\in KK(\mathcal{D}_x, \mathbb{C})$,
the isomorphism
\[(-\otimes I_{\mathcal{D}_x})\hat{\otimes}\nu_x : KK(\mathbb{C}, \mathcal{B}_x)\to KK(\mathcal{D}_x, \mathbb{C})\]
shows
\[KK((\mathcal{D}^\sharp)_x, \mathcal{O}_\infty)/\langle KK(d_x)\rangle \cong KK(\mathbb{C}, \mathcal{B}_x)/\langle KK(u_{\mathcal{B}_x})\rangle \cong K_1(C_{u_B})=KK(SA, \mathbb{C}).\]
Thus, we have isomorphisms
\[KK((\mathcal{D}^\sharp)_x, \mathbb{K})/\langle KK(d_x)\rangle=KK(SA, \mathbb{C})=KK(C_{u_A}, \mathbb{C})/\langle KK(e_A)\rangle.\]


The same argument as in proof of Thm. \ref{ca} gives a $KK$-equivalence $h\in KK(C_{u_A}, (\mathcal{D}^\sharp)_x)^{-1}$ with the following commutative diagram
\[\xymatrix{S\mathcal{O}_\infty\ar[r]^{\iota_x}&{C_d}_x\ar[r]&(\mathcal{D}^\sharp)_x\ar[r]^{d_x}&\mathcal{O}_\infty\\
S\ar[r]^{I_S\otimes u_A}\ar[u]&SA\ar[r]&C_{u_A}\ar[u]^{h}\ar[r]^{e_A}&\mathbb{C},\ar[u]
}\]
and one has $\sigma \in KK(SA, {C_d}_x)^{-1}$ making the above diagram commute.
Thanks to Thm. \ref{de},
one has a diagram
\[\xymatrix{SC(X)\ar[d]\ar@{.>}[r]^{I_S\otimes [p]_0}&S(SC_d)^\sharp\otimes\mathbb{K}\\
S^2SC(X)\otimes\mathcal{O}_\infty\ar[r]^{I_{S^2}\otimes\iota}&S^2C_d\ar[u]\\
SC(X)\otimes\mathcal{O}_\infty\ar[u]\ar[r]^{\iota}&C_d,\ar[u]
}\]
where all vertical maps are $KK_X$-equivalences and the broken arrow determines a properly infinite full projection $p\in (SC_d)^\sharp\otimes\mathbb{K}$.
The commutative diagram
\[\xymatrix{
S\ar[d]\ar[r]^{I_S\otimes [\pi_x(p)]_0}&S(C_{d_x})^\sharp\otimes\mathbb{K}\\
S\mathcal{O}_\infty\ar[r]^{\iota_x}&C_{d_x}\ar[u]\\
S\ar[u]\ar[r]^{I_S\otimes u_A}&SA\ar[u]^{\sigma}
}\]
shows
\[\pi_x(p)((SC_d)^\sharp\otimes\mathbb{K})_x\pi_x(p)\cong A.\] 
Thus, we obtain a locally trivial continuous $C(X)$-algebras $\mathcal{A}:= p((SC_d)^\sharp\otimes \mathbb{K})p$ with fiber $A$ and the following commutative diagram
\[\xymatrix{S^2\mathcal{A}\ar[r]&SC_{u_{\mathcal{A}}}\ar@{.>}[d]\ar[r]^{I_S\otimes e_\mathcal{A}}&SC(X)\ar[d]\ar[r]^{I_S\otimes u_\mathcal{A}}&S\mathcal{A}\\
SC_d\ar[u]\ar[r]&S(\mathcal{D}^\sharp)\ar[r]^{I_S\otimes d}&SC(X)\otimes\mathcal{O}_\infty\ar[r]^{\iota}&C_d\ar[u]\\
&S\mathcal{D}\ar[u]\ar[r]^{I_S\otimes((u_\mathcal{B}\otimes I_\mathcal{D})\hat{\otimes}\nu)}&SC(X)\ar[u]&
}\]
where every vertical arrow is a $KK_X$-equivalence and the broken arrow is given by Lem. \ref{RMN}.
Finally,
we obtain a $KK_X$-equivalence $\gamma \in KK_X(C_{u_\mathcal{A}}, \mathcal{D})$ with $\gamma\hat{\otimes}(u_\mathcal{B}\otimes I_\mathcal{D})\hat{\otimes}\nu=e_\mathcal{A}$.
It is straightforward to check $\mu\hat{\otimes}(\gamma^{-1}\otimes I_\mathcal{B})\hat{\otimes}(e_\mathcal{A}\otimes I_\mathcal{B})=u_\mathcal{B}$,
and $\mu\hat{\otimes}(\gamma^{-1}\otimes I_\mathcal{B})\in KK_X(C(X), C_{u_\mathcal{A}}\otimes_{C(X)}\mathcal{B})$ provides duality classes for $C_{u_{\mathcal{A}}}$ and $\mathcal{B}$.
Since the definition of $R_{A, B}$ is independent of the choice of duality classes, this implies $[\mathcal{B}]= {R_{A, B}}([\mathcal{A}])$.
\end{proof}
As mentioned in Rem. \ref{Q},
the categorical picture of Thm. \ref{ca} is not fully generalized for the case of $C(X)$-algebras.
So we may ask the following question.
\begin{que}
Does the equation $R_{A, B}^{-1}=R_{B, A}$ hold?
In other words,
is there a $KK_X$-equivalence $C_{u_\mathcal{B}}\to \mathcal{D}(\mathcal{A})$ making the diagram in Rem. \ref{Q} commutative?
\end{que}
\section{Appendix}\label{Ap}
We use the same notation and terminology as in Sec. \ref{nota}.
\begin{lem}\label{G1}
Let $G=G(p)\not =0$ be a finite Abel $p$-group with $L(G)=t$.
For $1\leq l\leq {\rm max}\{k\in I(G)\}$ and an element $g\in G$ of order $p^l$,
there exists $t\geq s\geq 1$
and we have an isomorphism
$$G\ni g \mapsto (0, \cdots, 0, p^{k_1-r_1}, \cdots, p^{k_s-r_s})\in (\mathbb{Z}_{p^{n_1}}\oplus\cdots\oplus\mathbb{Z}_{p^{n_{t-s}}})\oplus(\mathbb{Z}_{p^{k_1}}\oplus\cdots\oplus \mathbb{Z}_{p^{k_s}})$$
with $r_s=l, 0<r_i$, where the set $\{n_1, \cdots , n_{t-s}, k_1, \cdots , k_s\}$ is equal to $I(G)$. 
For $s\geq 2$,
the above $k_i, r_i$ satisfy
$$k_i<k_{i+1},\quad r_i<r_{i+1},\quad 0\leq k_1-r_1,\quad k_i-r_i<k_{i+1}-r_{i+1}.$$
\end{lem}
We note that the set $\{n_1, \cdots n_{t-s}\}$ is empty if $t=s$.
\begin{proof}
Every element $x\in \mathbb{Z}_{p^k}$ of order $p^r\; (r\leq k)$ is of the form $x=ap^{k-r}$ with $GCD(a, p)=1$, $a\in\mathbb{Z}$.
One has $a^{-1}\in \mathbb{Z}_{p^k}$ and the multiplication by $a^{-1}$ is an automorphism of $\mathbb{Z}_{p^k}$ that sends $x$ to $p^{k-r}$.
Thus we have an isomorphism $$G\ni g\mapsto (x_1,\cdots, x_t)\in \mathbb{Z}_{p^{k'_1}}\oplus\cdots\oplus\mathbb{Z}_{p^{k'_t}},$$ where every $x_i$ is of the form $x_i=p^{k'_i-R_i}$, $0\leq R_i\leq l$.

We may assume $k'_1\leq k'_2\leq\cdots\leq k'_t$,
and write $i_1:={\rm min}\{i \;|\;R_i=l\}$.
We will consider the following two cases :
$${I)}\colon {\rm there\;\; exists}\; i>i_1\; {\rm with}\; x_i=p^{k'_i-R_i}\not =0\in \mathbb{Z}_{p^{k'_i}},$$
$${II)}\colon {\rm there\;\; exists}\; i<i_1\; {\rm with}\; x_i=p^{k'_i-R_i}\not =0\in \mathbb{Z}_{p^{k'_i}},\;\; k'_i-R_i\geq k'_{i_1}-R_{i_1}.$$
If there exists $i>i_1$ satisfying I),
we have $$k'_i\geq k'_{i_1},\quad R_i\leq R_{i_1}=l,\quad (k'_i-  R_i)-(k'_{i_1}-R_{i_1})+k'_{i_1}\geq k'_i.$$
Thus, the following isomorphism
$$\mathbb{Z}_{p^{k'_{i_1}}}\oplus\mathbb{Z}_{p^{k'_i}}\ni (x, y)\mapsto (x, xp^{(k'_i-R_i)-(k'_{i_1}-R_{i_1})}+y)\in \mathbb{Z}_{p^{k'_{i_1}}}\oplus\mathbb{Z}_{p^{k'_i}}$$
is well-defined and this sends $(x_{i_1}, 0)$ to $(x_{i_1}, x_i)$.

If there exists $i<i_1$ satisfying II),
we have
$$(k'_i-R_i)-(k'_{i_1}-R_{i_1})+k'_{i_1}\geq k'_i,$$ 
and the following isomorphism is well-defined
$$\mathbb{Z}_{p^{k'_i}}\oplus\mathbb{Z}_{p^{k'_{i_1}}}\ni (y, x)\mapsto (xp^{(k'_i-R_i)-(k'_{i_1}-R_{i_1})}+y, x)\in \mathbb{Z}_{p^{k'_i}}\oplus\mathbb{Z}_{p^{k'_{i_1}}},$$
and this sends $(0, x_{i_1})$ to $(x_i, x_{i_1})$.

Now one can obtain an isomorphism
$$\mathbb{Z}_{p^{k'_1}}\oplus\cdots\oplus\mathbb{Z}_{p^{k'_t}}\ni (x_1, \cdots, x_t)\mapsto (y_1, \cdots, y_t) \in \mathbb{Z}_{p^{k'_1}}\oplus\cdots\oplus\mathbb{Z}_{p^{k'_t}}$$
with $y_i=p^{k'_i-R_i}, R_i\leq l$ such that there uniquely exists $i_1$ satisfying
$$R_{i_1}=l,\quad y_j=0 \;{\rm for}\; i_1<j,\quad R_j< R_{i_1}\; {\rm for}\; j<i_1,\quad k'_j-R_j<k'_{i_1}-R_{i_1}\; {\rm for}\;  y_j\not =0.$$
Since $k'_j-R_j<k'_{i_1}-R_{i_1}$ and $R_j< R_{i_1}$ imply $k'_j< k'_{i_1}$, we can prove the statement by applying the same argument for $(y_1, \cdots, y_{(i_1-1)})\in \mathbb{Z}_{p^{k'_1}}\oplus \cdots\oplus\mathbb{Z}_{p^{k'_{(i_1-1)}}}$ inductively.
\end{proof}
\begin{cor}\label{G2}
Let $g\in G=G(p)$, $I(G)=\{n_1,\cdots, n_{t-s}, k_1,\cdots, k_s\}$ and $r_i$ be as in Lem. \ref{G1}.
Then the quotient group $G/{\langle g\rangle}$ is isomorphic to
$$(\mathbb{Z}_{p^{n_1}}\oplus \cdots\oplus\mathbb{Z}_{p^{n_{t-s}}})\oplus(\mathbb{Z}_{p^{k_1-r_1}}\oplus\bigoplus_{i=2}^s \mathbb{Z}_{p^{k_i-r_i+r_{i-1}}})\;\;  (s\geq2),$$
$$(\mathbb{Z}_{p^{n_1}}\oplus \cdots\oplus\mathbb{Z}_{p^{n_{t-1}}})\oplus\mathbb{Z}_{p^{k_1-r_1}}\;\; (s=1).$$
In particular, we have $I(G)\cap I(G/\langle g\rangle)=\{n_1,\cdots , n_{t-s}\}$,
and $(G/\langle g\rangle, G)$ satisfies ($**$).
\end{cor}
\begin{proof}
It is easy to verify the case of $s=1$, and we prove the case of $s\geq 2$.
By Lem. \ref{G1},
we have $$G/\langle g\rangle\cong (\mathbb{Z}_{p^{n_1}}\oplus\cdots\oplus\mathbb{Z}_{p^{n_{t-s}}})\oplus((\mathbb{Z}_{p^{k_1}}\oplus\cdots\oplus\mathbb{Z}_{p^{k_s}})/\langle (p^{k_1-r_1},\cdots, p^{k_s-r_s})\rangle).$$
We write $h_i :=(y^i_1,\cdots, y^i_s)\in \mathbb{Z}_{p^{k_1}}\oplus\cdots\oplus\mathbb{Z}_{p^{k_s}}$ with $$y^i_j=0 \; (j<i),\quad y^i_i=1,\quad y^i_j=p^{(k_j-r_j)-(k_i-r_i)}\; (j>i).$$
Note that $\{h_i\}_{i=1}^s$ generate $\mathbb{Z}_{p^{k_1}}\oplus\cdots\oplus\mathbb{Z}_{p^{k_s}}$ and they satisfy the following equations
$$p^{k_1-r_1}h_1=(p^{k_1-r_1},\cdots, p^{k_s-r_s}),$$
$$p^{k_i-r_i+r_{i-1}}h_i=p^{r_{i-1}}(p^{k_1-r_1},\cdots, p^{k_s-r_s})$$
which provide the surjection $$\mathbb{Z}_{p^{k_1-r_1}}\oplus\bigoplus_{i=2}^s \mathbb{Z}_{p^{k_i-r_i+r_{i-1}}}\to(\mathbb{Z}_{p^{k_1}}\oplus\cdots\oplus\mathbb{Z}_{p^{k_s}})/\langle (p^{k_1-r_1},\cdots, p^{k_s-r_s})\rangle.$$
Since $|(\mathbb{Z}_{p^{k_1}}\oplus\cdots\oplus\mathbb{Z}_{p^{k_s}})/\langle (p^{k_1-r_1},\cdots, p^{k_s-r_s})\rangle|=p^{k_1+\cdots+k_s}/p^{r_s}$,
the above map is bijective.

By Lem. \ref{G1}, one has $k_{i-1}<k_i-r_i+r_{i-1}<k_i$ and this proves 
\begin{align*}
&I(G/\langle g\rangle)\cap I(G)\\
\subset&\{n_1,\cdots, n_{t-s}, k_1-r_1, k_2-r_2+r_1,\cdots, k_s-r_s+r_{s-1}\}\cap\{n_1,\cdots, n_{t-s}, k_1,\cdots,k_s\}\\
=&\{n_1, \cdots, n_{t-s}\}.
\end{align*}
Since $\{n_1,\cdots, n_{t-s}\}\subset I(G/\langle g\rangle)\cap I(G)$,
we have $\{n_1,\cdots, n_{t-s}\}= I(G/\langle g\rangle)\cap I(G)$.
\end{proof}
\begin{cor}\label{tort}
Let $g_m\in G=G(p), m=1, 2$ be elements of the finite Abel $p$-group satisfying $G/\langle g_1\rangle\cong G/\langle g_2\rangle$.
Then, there exists an automorphism of $G$ that sends $g_1$ to $g_2$.
\end{cor}
\begin{proof}
If $g_1=0$, the statement is trivial.
So we may assume that $g_1$ and $g_2$ have the same order $p^l, l\geq 1$.
By assumption, one has $$I(G)\cap I(G/\langle g_1 \rangle)=I(G)\cap I(G/\langle g_2 \rangle)=\{n_1, \cdots, n_{t-s}\},$$
where we write $I(G)=\{n_1, \cdots, n_{t-s}, k_1,\cdots, k_s\}$.
There exist  isomorphisms
$$G\ni g_m\mapsto (0,\cdots, 0, p^{k_1-r^m_1},\cdots, p^{k_s-r^m_s})\in (\mathbb{Z}_{p^{n_1}}\oplus\cdots \oplus\mathbb{Z}_{p^{n_{t-s}}})\oplus(\mathbb{Z}_{p^{k_1}}\oplus\cdots\oplus\mathbb{Z}_{p^{k_s}})$$
by Lem. \ref{G1} and Cor. \ref{G2},
and $k_i,\; r_i^m$ satisfy
$$k_i<k_{i+1}, \quad 0<r_i^m<r_{i+1}^m,\quad 0\leq k_i-r^m_i<k_{i+1}-r^m_{i+1},\quad r_s^m=l.$$
The assumption $G/\langle g_1\rangle\cong G/\langle g_2\rangle$ and Cor. \ref{G2} show $r_i^1=r_i^2$,
and this proves the statement.
\end{proof}

For a finite Abel $p$-group $G=G(p)$ and $g\in G,\;l\geq 0$,
the map $$G\ni x\mapsto [(x, 0)]\in (G\oplus \mathbb{Z})/\langle (g, p^l)\rangle$$ is injective and gives an exact sequence
$$0\to G\to (G\oplus\mathbb{Z})/\langle (g, p^l)\rangle\to \mathbb{Z}_{p^l}\to 0.$$
Thus, one has $|(G\oplus\mathbb{Z})/\langle (g, p^l)\rangle|=p^l|G|$,
and we will see that $(G, (G\oplus \mathbb{Z})/\langle (g, p^l)\rangle)$ satisfies ($**$).
If $l=0$,
one can easily find an isomorphism $G\oplus \mathbb{Z}\ni (g, 1)\mapsto (0, 1)\in G\oplus\mathbb{Z}$.

\begin{lem}\label{G3}
For $l>0$ and $g\not=0$,
we assume that there are no automorphisms of $G\oplus\mathbb{Z}$ which send $(g, p^l)$ to $(0, p^l)$.
Then, for some $s\geq 1$, there exists an isomorphism 
$$G\oplus\mathbb{Z}\to (\mathbb{Z}_{p^{n_1}}\oplus\cdots\oplus\mathbb{Z}_{p^{n_{t-s}}})\oplus (\mathbb{Z}_{p^{k_1}}\oplus\cdots\oplus\mathbb{Z}_{p^{k_s}})\oplus\mathbb{Z} $$
which sends  $(g, p^l)$ to $(0,\cdots 0, p^{k_1-r_1},\cdots, p^{k_s-r_s}, p^l)$
with $k_i-r_i<l, 0<r_i$,
and we have 
$$k_i<k_{i+1},\quad 0\leq k_i-r_i<k_{i+1}-r_{i+1},\quad 0< r_i<r_{i+1} \;\;{\rm for}\;\; s\geq 2.$$
Here, we write $I(G)=\{n_1,\cdots, n_{t-s}, k_1, \cdots, k_s\}$ as in Lem. \ref{G1}.
\end{lem}
\begin{proof}
Since $g\not =0$, the same argument as in Lem. \ref{G1} gives an isomorphism
$$G\oplus\mathbb{Z}\to (\mathbb{Z}_{p^{n_1}}\oplus\cdots\oplus\mathbb{Z}_{p^{n_{t-s}}})\oplus (\mathbb{Z}_{p^{k_1}}\oplus\cdots\oplus\mathbb{Z}_{p^{k_s}})\oplus \mathbb{Z}$$
which sends $(g, p^l)$ to $(0,\cdots 0, p^{k_1-r_1},\cdots, p^{k_s-r_s}, p^l)$ with $0<r_i$,
and we have
$$k_i<k_{i+1},\quad 0\leq k_i-r_i<k_{i+1}-r_{i+1},\quad 0<r_i<r_{i+1}$$
If $k_s-r_s\geq l$,
the following isomorphism
$$\mathbb{Z}_{p^{k_s}}\oplus\mathbb{Z}\ni (y, x)\mapsto (xp^{k_s-r_s-l}+y, x)\in \mathbb{Z}_{p^{k_s}}\oplus\mathbb{Z}$$
sends $(0, p^l)$ to $(p^{k_s-r_s}, p^l)$.
Therefore we may assume $k_i-r_i<k_s-r_s< l$, and this proves the statement.
\end{proof}
\begin{cor}\label{G4}
Let $(g, p^l)$, $k_i, r_i$ and $n_1,\cdots, n_{t-s}$ be as in Lem. \ref{G3}.
Then the quotient group $(G\oplus\mathbb{Z})/\langle (g, p^l)\rangle$ is isomorphic to
$$(\mathbb{Z}_{p^{n_1}}\oplus\cdots\oplus\mathbb{Z}_{p^{n_{t-s}}})\oplus(\mathbb{Z}_{p^{k_1-r_1}}\oplus(\bigoplus_{i=2}^s\mathbb{Z}_{p^{k_i-r_i+r_{i-1}}})\oplus\mathbb{Z}_{p^{l+r_s}})\;\; (s\geq 2),$$
$$(\mathbb{Z}_{p^{n_1}}\oplus\cdots\oplus\mathbb{Z}_{p^{n_{t-s}}})\oplus(\mathbb{Z}_{p^{k_1-r_1}}\oplus\mathbb{Z}_{p^{l+r_1}})\;\; (s=1).$$
If $s=1$,
one has $0\leq k_1-r_1<k_1<l+r_1$.
If $s\geq 2$,
one has $$0\leq k_1-r_1<k_1<k_2-r_2+r_1,\quad k_i-r_i+r_{i-1}<k_i<k_{i+1}-r_{i+1}+r_i,\quad k_s-r_s+r_{s-1}<k_s<l+r_s.$$
In particular,
we have
$$I(G)\cap I((G\oplus\mathbb{Z})/\langle (g, p^l)\rangle)=\{n_1, \cdots, n_{t-s}\}\not =I(G),$$
and $(G, (G\oplus\mathbb{Z})/\langle (g, p^l)\rangle)$ satisfies ($**$).
\end{cor}
\begin{proof}
We consider the case of $s\geq 2$,
and the proof for $s=1$ is the same.
We write $$(g, p^l)=(0,\cdots ,0, p^{k_1-r_1},\cdots, p^{k_s-r_s},p^l)\in (\mathbb{Z}_{p^{n_1}}\oplus\cdots\oplus\mathbb{Z}_{p^{n_{t-s}}})\oplus (\mathbb{Z}_{p^{k_1}}\oplus\cdots\oplus\mathbb{Z}_{p^{k_s}})\oplus\mathbb{Z},$$
and it is enough to show $$(\mathbb{Z}_{p^{k_1-r_1}}\oplus(\bigoplus_{i=2}^s\mathbb{Z}_{p^{k_i-r_i+r_{i-1}}})\oplus\mathbb{Z}_{p^{l+r_s}})\cong (\mathbb{Z}_{p^{k_1}}\oplus\cdots\oplus\mathbb{Z}_{p^{k_s}}\oplus\mathbb{Z})/\langle (p^{k_1-r_1},\cdots, p^{k_s-r_s}, p^l)\rangle.$$
Choose $u_i\in \mathbb{Z}_{p^{k_1}}\oplus\cdots\oplus\mathbb{Z}_{p^{k_s}}\oplus\mathbb{Z}$ by
$$u_i:=(h^i_1,\cdots h^i_s, p^{l-(k_i-r_i)}),$$
$$u_{s+1}:=(0,\cdots, 0, 1),$$
with
$$h^i_j=0 \;(j<i),\quad h^i_i=1,\quad h^i_j=p^{(k_j-r_j)-(k_i-r_i)} \;(j>i).$$
We have $$p^{k_1-r_1}u_1=(p^{k_1-r_1},\cdots, p^{k_s-r_s}, p^l),\quad p^{l+r_s}u_{s+1}=p^{r_s}(p^{k_1-r_1},\cdots, p^{k_s-r_s}, p^l)$$
and
$$p^{k_i-r_i+r_{i-1}}u_i=p^{r_{i-1}}(p^{k_1-r_1},\cdots, p^{k_s-r_s}, p^l)$$
for $2\leq i\leq s$.
Since $\{u_j\}_{j=1}^{s+1}$ generate $\mathbb{Z}_{p^{k_1}}\oplus\cdots\oplus\mathbb{Z}_{p^{k_s}}\oplus\mathbb{Z}$, one has a surjection
$$(\mathbb{Z}_{p^{k_1-r_1}}\oplus(\bigoplus_{i=2}^s\mathbb{Z}_{p^{k_i-r_i+r_{i-1}}})\oplus\mathbb{Z}_{p^{l+r_s}})\to  (\mathbb{Z}_{p^{k_1}}\oplus\cdots\oplus\mathbb{Z}_{p^{k_s}}\oplus\mathbb{Z})/\langle (p^{k_1-r_1},\cdots, p^{k_s-r_s}, p^l)\rangle,$$
which is bijective by the equation 
\begin{align*}
&|(\mathbb{Z}_{p^{k_1}}\oplus\cdots\oplus\mathbb{Z}_{p^{k_s}}\oplus\mathbb{Z})/\langle (p^{k_1-r_1},\cdots, p^{k_s-r_s}, p^l)\rangle|\\
=&p^l|\mathbb{Z}_{p^{k_1}}\oplus\cdots\oplus\mathbb{Z}_{p^{k_s}}|\\
=&p^{k_1+\cdots + k_s+l}.
\end{align*}
\end{proof}
If $(g, p^l)$ is sent to $(0, p^l)$,
the quotient group $(G\oplus\mathbb{Z})/\langle (g, p^l)\rangle$ is equal to $G\oplus \mathbb{Z}_{p^l}$,
and the pair $(G, (G\oplus\mathbb{Z})/\langle (g, p^l)\rangle)$ satisfies ($**$).
\begin{cor}\label{kmc2}
Let $G=G(p)$ be a finite Abel p-group.
For $(g_m, p^l)\in G\oplus\mathbb{Z}, (m=1, 2,\; l\geq 0)$,
there exists an automorphism $\Theta : G\oplus\mathbb{Z}\to G\oplus\mathbb{Z}$ with $\Theta((g_1, p^l))=(g_2, p^l)$ if and only if we have $(G\oplus\mathbb{Z})/\langle (g_1, p^l)\rangle\cong(G\oplus\mathbb{Z})/\langle (g_2, p^l)\rangle$.
\end{cor}
\begin{proof}
We may assume $l>0$.
If there exists an automorphism of $G\oplus\mathbb{Z}$ which sends $(g_1, p^l)$ to $(0, p^l)$,
the quotient group $(G\oplus\mathbb{Z})/\langle (g_1, p^l)\rangle$ is isomorphic to $G\oplus\mathbb{Z}_{p^l}$,
and we have 
\begin{align*}
&I(G)\\
=&I(G)\cap I((G\oplus\mathbb{Z})/\langle (g_1, p^l)\rangle)\\
=&I(G)\cap I((G\oplus\mathbb{Z})/\langle (g_2, p^l)\rangle)
\end{align*}
and Cor. \ref{G4} shows that $(g_2, p^l)$ must be sent to $(0, p^l)$ by an automorphism.
So we may assume that there are no automorphisms that send $(g_m, p^l)$ to $(0, p^l)$.

Since $$I(G)\cap I((G\oplus\mathbb{Z})/\langle (g_1, p^l)\rangle)=I(G)\cap I((G\oplus\mathbb{Z})/\langle (g_2, p^l)\rangle),$$
Cor. \ref{G4} and Lem. \ref{G3} give  an isomorphism which sends $(g_m, p^l)\in G\oplus\mathbb{Z}$ to the following element :
$$(0,\cdots ,0, p^{k_1-r^m_1},\cdots, p^{k_s-r^m_s}, p^l)\in (\mathbb{Z}_{p^{n_1}}\oplus\cdots\oplus\mathbb{Z}_{p^{n_{t-s}}})\oplus (\mathbb{Z}_{p^{k_1}}\oplus\cdots\oplus\mathbb{Z}_{p^{k_s}})\oplus \mathbb{Z},$$
and Cor. \ref{G4} proves $r^1_i=r^2_i$ for $1\leq i\leq s$.
\end{proof}
\begin{lem}\label{mou}
Let $G=\bigoplus_p G(p)$ be a finite Abel group with $g=(g_p)_p\in G$, and let $n=\prod_p p^{n_p}$ be the prime decomposition of $n\in\mathbb{Z}_{\geq 1}$.
Then, the following hold $\colon$

\end{lem}
\begin{proof}
We write $n=\prod_pp^{n_p}=p^{n_p}r_p$,
and let $R_p$ denote $\prod_{q\not =p}ord(g_q)$.
One has $\langle g\rangle= \bigoplus_p\langle g_p\rangle\subset G$ by $GCD(R_p, p)=1$,
and this proves the statement 1.

Next, we check the statement 2.
For a finite set $F:=\{p\; |\; G(p)\not=0, \;{\rm or}\; n_p\not=0\}$,
one has $GCD\{R_pr_p\; |\; p\in F\}=1$,
which implies the surjectivity of the map
$$\bigoplus_{p\in F}(G(p)\oplus\mathbb{Z})\ni ((h_p, x_p))_p\mapsto ((R_ph_p)_p, (\sum_p R_pr_px_p))\in (\bigoplus_{p\in F}G(p))\oplus\mathbb{Z}.$$
The above map induces a well-defined surjection
$$\bigoplus_{p\in F}((G(p)\oplus\mathbb{Z})/\langle (g_p, p^{n_p})\rangle)\to ((\bigoplus_{p\in F}G(p))\oplus\mathbb{Z})/\langle ((g_p)_p, n)\rangle.$$
Since $\prod_{p\in F}p^{n_p}|G(p)|=n|G|$,
the above map is injective,
and this proves the statement 2.
\end{proof}

We prove Prop. \ref{vn} and \ref{kmc} using  the above preliminaries.
\begin{proof}[{Proof of Proposition \ref{vn}}]
First, we discuss the case that $[1_A]_0\in K_0(A)$ is a torsion element.
We write $K_0(A)=(\bigoplus_p K_0(A)(p))\oplus \mathbb{Z}^F\ni [1_A]_0=((g_p)_p, 0,\cdots ,0)$,
and it is easy to check that $K_1(C_{u_A})(p)=K_0(A)(p)/\langle g_p\rangle$.
Thus, Cor. \ref{G2} proves that $(K_1(C_{u_A})(p), K_0(A)(p))$ satisfies ($**$).

Next,
we consider the case that $[1_A]_0\in K_0(A)$ is not a torsion element.
Then,
we may assume $$[1_A]_0=((g_p)_p, n, 0,\cdots, 0)\in K_0(A)=(\bigoplus_p K_0(A)(p))\oplus \mathbb{Z}^{1+F}$$
for some $n\geq 1$.
For the prime decomposition $n=\prod_pp^{n_p}$,
Lem. \ref{mou} proves $K_1(C_{u_A})(p)=(K_0(A)(p)\oplus\mathbb{Z})/\langle (g_p, p^{n_p})\rangle$,
and Cor. \ref{G4} shows that $(K_0(A)(p), K_1(C_{u_A})(p))$ satisfies ($**$).
\end{proof}
\begin{proof}[{Proof of Proposition \ref{kmc}}]
Let $G$ be a finitely generated Abel group, and let $g_1, g_2\in G$ satisfy $G/\langle g_1\rangle\cong G/\langle g_2\rangle$.
Then, $g_1$ is a torsion element (resp. a non-torsion element) if and only if so is $g_2$ by the comparison of ranks of $G/\langle g_i\rangle$.
By comparison of the orders of torsion parts of $G/\langle g_i\rangle$, we may assume that $G=H\oplus \mathbb{Z}$ for a finite Abel group $H$ and $g_i=(h_i, n)\in H\oplus \mathbb{Z}$ for some $n\in\mathbb{Z}_{\geq 0}$.
We write $h_i=({h^i}_p)_p\in H=\bigoplus_p H(p)$, $n=\prod_pp^{n_p}$,
and Lem. \ref{mou} shows
$$(G/\langle g_i\rangle)(p)\cong H(p)/\langle {h^i}_p\rangle\quad (n=0),$$
$$(G/\langle g_i\rangle)(p)\cong (H(p)\oplus \mathbb{Z})/\langle({h^i}_p, p^{n_p})\rangle\quad (n\not=0).$$
We write $r_p=\prod_{q\not=p}q^{n_q}=n/p^{n_p}$, and denote by $r_p^{-1}$ the inverse map of $H(p)\ni x\mapsto r_px\in H(p)$.
Now one has
$$H(p)/\langle h^1_p\rangle\cong H(p)/\langle h^2_p\rangle \quad (n=0),$$
$$(H(p)\oplus\mathbb{Z})/\langle (r_p^{-1}(h^1_p), p^{n_p})\rangle \cong (H(p)\oplus\mathbb{Z})/\langle (r_p^{-1}(h^2_p), p^{n_p})\rangle\quad (n\not=0),$$
and Cor. \ref{tort} proves the statement for the case $n=0$.
Applying Cor. \ref{kmc2} to the case $n\not=0$,
one has an automorphism $\theta$ of $H(p)\oplus \mathbb{Z}$ with $$\theta((r_p^{-1}(h^1_p), p^{n_p}))=(r_p^{-1}(h^2_p), p^{n_p}).$$
Since the map $(\bigoplus_{q\not=p}{\rm id}_{{H(q)}})\oplus \theta$ sends $((h^1_q)_q, (h^1_p, n))\in (\bigoplus_p H(p))\oplus\mathbb{Z}$ to $((h^1_q)_q, (h^2_p, n))$,
one can obtain desired isomorphism $H\oplus\mathbb{Z}\ni ((h^1_p)_p, n)\mapsto ((h^2_p)_p, n)\in H\oplus\mathbb{Z}$ by applying the same argument for every $q$.
\end{proof}

On behalf of all authors, the corresponding author states that there is no conflict of interest.
This paper has no associated data.

\end{document}